\documentclass[11pt, oneside]{article}
\usepackage{amsfonts}
\usepackage{mathrsfs}
 \usepackage{color}
\usepackage{latexsym}
\usepackage{amssymb}
\usepackage{amsmath}
\usepackage{enumerate}
\usepackage{amsthm}
\usepackage{indentfirst}
\usepackage{mathtools}
\usepackage{tikz}
\usepackage{psfrag}
\usepackage{graphicx}
\usepackage{bm}
\usepackage[rflt]{floatflt}
\usepackage{float}
\usepackage{authblk}
\usepackage{ulem}
\usepackage[square, comma, sort&compress, numbers]{natbib}
\usepackage{verbatim}
\usepackage{amsthm}
\usepackage{esint}

\numberwithin{equation}{section}
\allowdisplaybreaks

\newenvironment{proof2.1}{\medskip\noindent{\bf Proof of the Theorem 2.1:}\enspace}{\hfill \qed \newline \medskip}

\newenvironment{proof2.2}{\medskip\noindent{\bf Proof of the Theorem 2.2:}\enspace}{\hfill \qed \newline \medskip}

\newtheorem{theorem}{\color{black}\indent Theorem}[section]
\newtheorem{lemma}{\color{black}\indent Lemma}[section]
\newtheorem{assumption}{\color{black}\indent Assumption}[section]
\newtheorem{proposition}{\color{black}\indent Proposition}[section]
\newtheorem{definition}{\color{black}\indent Definition}[section]
\newtheorem{remark}{\color{black}\indent Remark}[section]

\newtheorem*{acknowledgment}{Acknowledgment}

\DeclareRobustCommand{\rchi}{{\mathpalette\irchi\relax}}
\newcommand{\irchi}[2]{\raisebox{\depth}{$#1\chi$}}

\usepackage{amssymb,amsmath}
\begin{document}

\title{H\"older and Harnack estimates for\\integro-differential operators with kernels of measure}

\author[1,2]{Jingya Chen\thanks{\hspace*{4mm}{Email addresses: 
 jingya21@mails.jlu.edu.cn}\newline \hspace*{4mm}{{\it Mathematics Subject Classiffcation (2020)}: 35B45, 35B65, 35R11, 47G20}\newline\hspace*{4mm}{{\it Key words and phrases}: nonlocal, H\"older regularity, Harnack estimate}}}
\affil[1]{School of Mathematics, Jilin University, Qianjin Str. 2699, Changchun, Jilin Province 130012, China}
\affil[2]{Fachbereich Mathematik, Universit\"at Salzburg, Hellbrunner Str. 34, 5020 Salzburg, Austria}
\renewcommand*{\Affilfont}{\small\it}
\date{} \maketitle
\vspace{-20pt}
\begin{abstract}
We establish H\"older and Harnack estimates for weak solutions of a class of elliptic nonlocal equations that are modeled on integro-differential operators with kernels of measure. The approach is of De Giorgi-type, as developed by DiBenedetto, Gianazza and Vespri in a local setting. Our results generalize the work by Dyda and Kassmann (2020).
\end{abstract}
\section{Introduction}
The aim of this work is to prove regularity estimates for local weak solutions of elliptic equations of type
\begin{equation}\label{equa}
\mathscr{L}u\equiv\text{p.v.}\int_{\mathbb{R}^N}\big(u(x)-u(y)\big)\mu(x,dy)=0\quad\text{in}~E.
\end{equation}
Here, $E\subset\mathbb{R}^N~(N\geqslant 2)$ is a bounded domain, the symbol \text{p.v.} means ``in the principal value sense", whereas $(\mu(x,\cdot))_{x\in\mathbb{R}^N}$ is a family of measures on $\mathbb{R}^N$ whose properties will be specified next.
\par
Given a  general family of measures that satisfy basic structural conditions such as symmetry, a Poincar\'e-type inequality and a proper decay estimate, locally bounded, local weak solutions of (\ref{equa}) are locally H\"older continuous. Under the same assumptions on the measures, we obtain a weak Harnack inequality for local weak super-solutions of (\ref{equa}) as well. If $\mu$ is absolutely continuous with respect to the Lebesgue measure and the density function satisfies a certain ``uniform jump and symmetry" condition, then we conclude a boundedness estimate for local weak solutions, that will lead to a full Harnack inequality.
\par
Our concept of solutions involves quadratic forms related to $\mu(x,dy)$. The following two subsections are devoted to the set-up, assumptions and the statements of our main results.
\subsection{Definitions and Notation}
Throughout this note, we will use $B_r(x_0)$ to denote the ball of radius $r$ and center $x_0$ in $\mathbb{R}^N$. When $x_0$ is the origin, we will omit the center $x_0$ from the symbol for simplicity.
\par
Let $\mu=(\mu(x,\cdot))_{x\in\mathbb{R}^N}$ be a family of Borel measures on $\mathbb{R}^N$. We say $\Omega$ is $\mu$-measurable if it is $\mu(x,dy)$-measurable for all $x\in\mathbb{R}^N$. They are symmetric in the following sense.
\begin{assumption}
For any $\mu$-measurable sets $\Omega_1$ and $\Omega_2$ with $\Omega_1\cap\Omega_2=\varnothing$, we have that $x\mapsto\int_{\Omega_i}\mu(x,dy),~i=1,2$ are Lebesgue measurable and
\begin{equation}\label{S}
\int_{\Omega_1}\int_{\Omega_2}\mu(x,dy)dx=\int_{\Omega_2}\int_{\Omega_1}\mu(x,dy)dx.\tag{S}
\end{equation}
\end{assumption}
For a given family $\mu$ we consider the following quadratic forms on $L^2(D)\times L^2(D)$, where $D\subseteq\mathbb{R}^N$ is some open set:
\[\mathcal{E}_D^\mu(u,\varphi)=\int_D\int_D(u(y)-u(x))(\varphi(y)-\varphi(x))\,\mu(x,dy)dx.\]

\begin{definition}[function spaces]\upshape
We say $u\in H_{\text{loc}}^\mu(E)$ if $\mathcal{E}_D^\mu(u,u)$ is finite for all $\bar{D}\subseteq E$.
Let $H^\mu(\mathbb{R}^N)$ be the vector space of functions $u\in L^2(\mathbb{R}^N)$ such that $\mathcal{E}^\mu(u,u)=\mathcal{E}_{\mathbb{R}^N}^\mu(u,u)$ is finite.
\end{definition}

\par

\begin{definition}[solution concept]\upshape
A Lebesgue measurable function $u:\mathbb{R}^N\rightarrow\mathbb{R}$ is a local weak sub(super)-solution of (\ref{equa}), if it satisfies $u\in H_{\text{loc}}^\mu(E)$,
\[\sup_{x\in B_{\frac78R}(x_0)}\int_{\mathbb{R}^N\backslash B_R(x_0)}|u(y)|\,\mu(x,dy)<\infty,\quad\text{for any}~B_R(x_0)\subset E\]
and
\begin{equation}\label{def-solution}
\mathcal{E}^{\mu}(u,\varphi)\leqslant(\geqslant) 0
\end{equation}
for every non-negative $\varphi\in H^\mu(\mathbb{R}^N)$ with supp\,$u\subseteq E$.
A function $u$ that is both a local weak sub-solution and a local weak super-solution of (\ref{equa}) is a local weak solution.
\end{definition}
Note that $\mathcal{E}^{\mu}(u,\varphi)$ is finite as can be seen from the decomposition
\begin{align*}
\mathcal{E}^{\mu}(u,\varphi)=&\int_{D}\int_{D}(u(y)-u(x))(\varphi(y)-\varphi(x))\,\mu(x,dy)dx\\
&+2\int_{D}\int_{D^c}(u(y)-u(x))(\varphi(y)-\varphi(x))\,\mu(x,dy)dx,
\end{align*}
where we used (S), and from the function spaces $u$ and $\varphi$ belong to.
\par
A central object of our study is the \textit{nonlocal tail} of a sub(super)-solution $u$ given by
\begin{equation*}
{\rm Tail}(u;x_0,r_1,r_2):=r_2^\alpha\underset{x\in B_{r_1}(x_0)}{\text{sup}}\int_{\mathbb{R}^N\backslash B_{r_2}(x_0)}u(y)\,\mu(x,dy),
\end{equation*}
where $x_0\in\mathbb{R}^N$ and $B_{r_1}(x_0)\subset B_{r_2}(x_0)\subset E$. When $x_0=0$, we omit it from the notation. According to the notion of solutions, the tail is finite if $r_1\leqslant\frac78r_2$.
\subsection{Assumptions and Main results}
In this subsection, we list the other assumptions on $\mu$. Throughout this paper, we take $\alpha\in(0,2)$. Firstly, a Poincar\'e-type inequality holds true.
\begin{assumption}\label{AA}
    Let $\mathfrak{C}_1\geqslant 1$. For every ball $B_{\rho}(x_0)\subset E$ and every $v\in H_{\mathrm{loc}}^{\mu}(E)$,
    \begin{equation}\label{A-poin}
        \int_{B_\rho(x_0)}(v(x)-[v]_{B_\rho(x_0)})^2dx\leqslant \mathfrak{C}_1\rho^\alpha\mathcal{E}_{B_\rho(x_0)}^{\mu}(v,v),\tag{H$_1$}
    \end{equation}
    where $$[v]_{B_\rho(x_0)}=\fint_{B_\rho(x_0)}v(x)dx.$$
\end{assumption}
We also need an additional assumption on the decay of the measures $\mu$ considered.
\begin{assumption}\label{DD}
Let $\mathfrak{C}_2\geqslant 1$. For every $x_0\in\mathbb{R}^N$ and $r>0$,
\begin{equation}\label{D}
\int_{\mathbb{R}^N\backslash B_{r}(x_0)}\mu(x_0,dy)\leqslant\frac{\mathfrak{C}_2}{r^\alpha}.\tag{H$_2$}
\end{equation}
\end{assumption}
\par
\begin{remark}\upshape
An important example satisfying the above conditions (\ref{S}), (\ref{A-poin}) and (\ref{D}) is given by
\begin{equation*}
\mu_{\alpha}(x,dy)=(2-\alpha)\frac{1}{|x-y|^{N+\alpha}}\,dy.
\end{equation*}
We denote by $H^{\alpha/2}(E)$ the usual Sobolev space of fractional order $\alpha/2\in(0,1)$ with the norm
\[\|u\|_{H^{\alpha/2}(E)}=(\|u\|_{L^2(E)}^2+\mathcal{E}_{E}^{\mu_\alpha}(u,u))^{\frac{1}{2}}.\]
In particular, Assumption~\ref{AA} plays the role of an ellipticity condition. It naturally holds if for every ball $B_\rho(x_0)\subset E$ and every $v\in H^{\alpha/2}(B_\rho(x_0))\cap H_{\text{loc}}^\mu(E)$,
\begin{equation*}\label{A1}
\mathcal{E}_{B_\rho(x_0)}^{\mu_\alpha}(v,v)\leqslant C\mathcal{E}_{B_\rho(x_0)}^{\mu}(v,v)
\end{equation*}
for some $C>0$. This is an assumption used in \cite{kass1}. It is also assumed that the reverse estimate holds in \cite{kass1}. We find, however, this is redundant.
\end{remark}
\par
Now, let us formulate some of our main results.
\begin{theorem}[H\"older regularity]\label{Holder}
Let $\mu$ satisfy (\ref{S}), (\ref{A-poin}) and (\ref{D}). Let $u$ be a locally bounded, local weak solution of (\ref{equa}) in $E$. Then, $u$ is locally H\"older continuous in $E$. More precisely, there exist constants $\gamma>1$ and $\alpha_*\in(0,1)$ that can be determined a priori only in terms of the \textit{data}\,$\{\mathfrak{C}_1,\mathfrak{C}_2,N,\alpha\}$, such that for any $0<\rho<R,$ there holds
\begin{equation*}
\underset{B_\rho(x_0)}{\rm{osc}}~u\leqslant\gamma\omega\left(\frac{\rho}{R}\right)^{\alpha_*},
\end{equation*}
provided the ball $B_R(x_0)$ is included in $E$, where
\begin{equation*}
\omega=2\underset{B_R(x_0)}{\rm{sup}}~|u|+{\rm Tail}(|u|;x_0,\tfrac78R,R).
\end{equation*}
\end{theorem}
\begin{theorem}[weak Harnack estimate]\label{weak-harn}
Let $\mu$ satisfy (\ref{S}), (\ref{A-poin}) and (\ref{D}). Let $u$ be a local weak super-solution of (\ref{equa}) in $E$, such that $u\geqslant 0$ in $B_R(x_0)$. Then, for any $0<r<R/32$, there holds
\begin{equation*}
\inf_{B_{2r}(x_0)}u\geqslant \eta\bigg(\fint_{B_r(x_0)}u^{\varepsilon}\,dx\bigg)^{\frac{1}{\varepsilon}}-\left(\frac{r}{R}\right)^\alpha{\rm Tail}(u_-;x_0,\tfrac78R,R),
\end{equation*}
provided the ball $B_R(x_0)$ is included in $E$, where the constants $\varepsilon,\eta\in(0,1)$ depend only on $\mathfrak{C}_1,\mathfrak{C}_2,N$ and $\alpha$.
\end{theorem}
Under the mere conditions (\ref{S}), (\ref{A-poin}) and (\ref{D}), a full Harnack estimate generally cannot be reached; see \cite{bogdan,schu} for counterexamples. In what follows, we adopt a standard assumption, called (\ref{UJS}), in order to establish a full Harnack estimate.
\par
We assume that $\mu$ can be written as
\begin{equation}\label{mu}
\mu(x,dy)=K(x,y)\,dy
\end{equation}
where $K:\mathbb{R}^N\times\mathbb{R}^N\rightarrow[0,\infty)$ is a measurable function that is symmetric in the sense that $K(x,y)=K(y,x)$ for all $x,y\in\mathbb{R}^N$, and $K(x,y)$ satisfies (\ref{UJS}) as follows.
\begin{assumption}
There exists $\hat{c}>0$ such that for every $x,y\in\mathbb{R}^N$ and every $r\leqslant \frac{|x-y|}{2}$ we have
\begin{equation}\label{UJS}
K(x,y)\leqslant\hat{c}\fint_{B_r(x)}K(z,y)\,dz.\tag{UJS}
\end{equation}
\end{assumption}
\noindent Apparently $\mu_\alpha$ satisfies (\ref{UJS}) condition. We can find some more examples that satisfy (\ref{UJS}) in \cite[Chapter~12]{schu}. Here, we point out that the (\ref{UJS}) assumption together with (\ref{D}) implies a pointwise upper bound of $K$; see Lemma~\ref{point-upper}.
\par
Now we are ready to formulate the other two main results in this paper.

\begin{theorem}[Harnack estimate]\label{full-harn}
Let $\mu$ be represented by (\ref{mu}) and satisfy (\ref{S}), (\ref{A-poin}), (\ref{D}) and (\ref{UJS}). Let $u$ be a local weak solution of (\ref{equa}) in $E$, such that $u\geqslant 0$ in $B_R(x_0)$. Then, for any $0<r<R/32$, there holds
    \begin{equation*}
        \sup_{B_r(x_0)}u\leqslant\gamma\inf_{B_{r}(x_0)}u+\gamma \left(\frac{r}{R}\right)^\alpha{\rm Tail}(u_-;x_0,\tfrac78R,R),
    \end{equation*}
    provided the ball $B_R(x_0)$ is included in $E$, where the constant $\gamma>1$ depends only on $\mathfrak{C}_1,\mathfrak{C}_2,N,\alpha$ and $\hat{c}$.
\end{theorem}
In pursuit of the Harnack estimate, we also establish a boundedness estimate for sub-solutions.
\begin{theorem}[Local boundedness estimate]\label{bounded2}
Let $\mu$ be represented by (\ref{mu}) and satisfy (\ref{S}), (\ref{A-poin}), (\ref{D}) and (\ref{UJS}). Let $B_R(x_0)\subset E$. Let $u$ be a local weak sub-solution of (\ref{equa}) in $E$, such that $u\geqslant 0$ in $B_R(x_0)$. Then, for any $0<r<R$,
    \begin{equation*}
        \sup_{B_{\frac{r}{2}}(x_0)}u\leqslant\delta r^\alpha\int_{\mathbb R^N\backslash B_{\frac{r}{2}}(x_0)}\frac{u_+(x)}{|x-x_0|^{N+\alpha}}\,dx+\gamma\frac{1}{\delta^{\frac{N}{2\alpha}}}\bigg(\fint_{B_r(x_0)}u_+^2\,dx\bigg)^{\frac{1}{2}},
    \end{equation*}
    for any $\delta\in(0,1]$, where the constant $\gamma$ depends only on $\mathfrak{C}_1,\mathfrak{C}_2,N,\alpha$ and $\hat{c}$.
\end{theorem}
Throughout this article, we refer to $\{\mathfrak{C}_1,\mathfrak{C}_2,N,\alpha\}$ as {\it data}. In various estimates, we will use $\gamma$ as a generic constant.

\subsection{Novelty and significance}
The local regularity theory for elliptic equations driven by non-local operators has become a highly active field of research in recent years. In the article~\cite{kass1}, the authors considered a kind of general integro-differential operators with kernels of measure which are allowed to be singular. They established regularity results including a weak Harnack inequality for super-solutions and H\"older regularity for globally bounded weak solutions of (\ref{equa}). On general metric measure spaces with the volume doubling property, Harnack inequalities and H\"older estimates of elliptic and parabolic equations were established in~\cite{chen-kuma1,chen-kuma2,hu-yu} in connection with Dirichlet forms.
\par
Our study is motivated by \cite{kass1} and clarifies some assumptions on the measures $(\mu(x,\cdot))_{x\in\mathbb{R}^N}$ that lead to regularity estimates. In particular, we find that a Poincar\'e-type inequality and a decay estimate of the measures are sufficient for H\"older regularity and weak  Harnack estimates to hold. In addition, if we assume a so-called (\ref{UJS}) condition, then Harnack estimates follow. This kind of simplification is probably known to experts. However, we do not find it in print.
\par
Different from the Moser-type approach adopted in \cite{kass1}, the conclusion in this paper mainly follows from DeGiorgi's approach, developed by DiBenedetto, Gianazza and Vespri in local settings. We also follow some ideas by Liao in \cite{liao2} to complete our proofs. This approach differs from the Moser iteration and circumvents a large amount of algebraic inequalities. In fact, we find that the regularity properties are encoded in the energy estimates in Proposition~\ref{DG-class} and logarithmic estimate in Proposition~\ref{log-esti}.
\par
In regularity theory one of the essential tools is the so-called \textit{expansion of positivity}, which translates measure theoretical estimates of positivity into a pointwise one. Inspired by \cite{casrto}, a logarithmic estimate plays an important role in the derivation of the expansion of positivity, which is the key step for obtaining H\"older estimate and weak Harnack estimate. The H\"older regularity (Theorem~\ref{Holder}) provides an extension of the analogous result in~\cite{kass1} to the general measure framework as it weakens the conditions on the measure $\mu$ and manages the tail.
\par
Theorem~\ref{weak-harn} concerns a weak Harnack inequality. Its proof hinges on an improved expansion of positivity, avoiding crossover estimates based on BMO estimates of $\log\,u$. A classic method to prove the improved expansion of positivity lemma is to apply Krylov-Safanov covering arguments, see for instance \cite{cozzi,casrto1}. This dates back to the pioneering work of DiBenedetto and Trudinger \cite{dibe-trud}. In our work, we provide an alternative approach via the clustering technique of DiBenedetto, Gianazza and Vespri stated in Lemma~\ref{cluster}. This more geometric approach has the potential to be generalized to study different nonlinear, nonlocal operators, also parabolic ones, See in this connection \cite{cassa}.
\par
To establish a full Harnack inequality, we additionally assume that $\mu$ can be represented as in (\ref{mu}) and satisfies (\ref{UJS}). A key step towards the Harnack inequality is a boundedness estimate stated in Theorem~\ref{bounded2}. Joining the boundedness estimate and the weak Harnack inequality, we conclude a full Harnack inequality via an interpolation argument. To the best of the author's knowledge, the (\ref{UJS}) assumption was first introduced by Barlow, Bass and Kumagai in \cite{barlow}, where it appeared in the setting of discrete graphs. See \cite{chen-kim} for a general setting of metric measure spaces. Our (\ref{UJS}) assumption is a localized version, which is also used by Schulze in \cite{schu}.
\subsection{Structure of the paper}
This article is structured as follows. In Section~2, we provide some preliminary results. In Section~3 we deduce energy estimates for local weak sub(super)-solutions and an expansion of positivity lemma for local weak super-solutions. Section~4 provides the proof of H\"older regularity, whereas the weak Harnack inequality is proved in Section~5. In Section~6, we prove local boundedness estimates for weak solutions or sub-solutions. Finally, in Section~7 we develop a full Harnack inequality.
\begin{acknowledgment}
The author completed this paper during a research stay at Fachbereich Mathematik, Universit\"at Salzburg. The author thanks Prof. B\"ogelein and Dr. Liao for helpful discussions on the subject of this paper. The author thanks Dr. Weidner for discussing technical details related to Faber-Krahn inequality.
\end{acknowledgment}
\section{Auxiliary results}\label{notations}
The purpose of this section is to provide several implicant of (\ref{S}), (\ref{A-poin}) and (\ref{D}). In addition, we establish a clustering lemma. The fact that the Poincar\'e inequality implies the Faber-Krahn inequality is well-known in local settings; see for example \cite{sal}. However, we do not find a detailed proof for our purpose in print. For the convenience of readers, we state and prove the relative lemma applicable to this paper.
\begin{lemma}\label{poin->FK}
If (\ref{A-poin}) holds, $v\in H_{\mathrm{loc}}^\mu(B_\rho)$ and $\mathrm{supp}\,v\subset\Omega_0\subset B_\rho$, then there exists a constant $c_0$ only depending on $\mathfrak{C}_1,N$ and $\alpha$, such that
\begin{equation}\label{FK}
c_0\mathcal{E}_{\mathbb{R}^N}^{\mu}(v,v)\geqslant |\Omega_0|^{-\frac{\alpha}{N}}\int_{\Omega_0}|v|^2\,dx.\tag{FK}
\end{equation}
\end{lemma}
\begin{proof}
Let $f\in H_{\text{loc}}^\mu(B_\rho)$ satisfy $\text{supp}\,f\subset B_\rho$ and set $f_r(x):=\fint_{B_r(x)}f(y)dy$ for any $r>0$ and $x\in\mathbb{R}^N$. It follows from triangle's inequality that
\begin{equation}\label{est5-1}
\|f\|_{L^2(B_\rho)}^2\leqslant2\|f-f_r\|_{L^2(B_\rho)}^2+2\|f_r\|_{L^2(B_\rho)}^2.
\end{equation}
Fix $r>0$ and let $\{x_n\}_{n\in\mathbb{N}}$ be a countable set in $\mathbb{R}^N$ satisfying
\begin{equation*}
\bigcup_{n=1}^\infty B_r(x_n)=\mathbb{R}^N \quad\text{and}\quad \sum_{n\in\mathbb{N}}{\rchi}_{B_{2r}(x_n)}(y)\leqslant\theta_N,~\forall y\in\mathbb{R}^N,
\end{equation*}
where $\theta_N$ depends only on $N$. Besicovitch's covering theorem ensures that there exists such a countable set $\{x_n\}$.
Then, by (\ref{A-poin}) we have
\begin{align}\label{est5-2}
\|f-f_r\|_{L^2(B_\rho)}^2
=&~|B_r|^{-2}\int_{B_\rho}\Big(\int_{B_r(x)}|f(x)-f(y)|\,dy\Big)^2dx\nonumber\\
\leqslant&~|B_r|^{-2}\sum_{n=1}^\infty\int_{B_r(x_n)}\Big(\int_{B_r(x)}|f(x)-f(y)|\,dy\Big)^2dx\nonumber\\
\leqslant&~|B_r|^{-2}\sum_{n=1}^\infty\int_{B_r(x_n)}\Big(\int_{B_{2r}(x_n)}|f(x)-f(y)|\,dy\Big)^2dx\nonumber\\
\leqslant&~\mathfrak{C}_1|B_r|^{-2}|B_{2r}|^2\sum_{n=1}^\infty(2r)^{\alpha}\mathcal{E}_{B_{2r}(x_n)}^\mu(f,f)\nonumber\\
\leqslant&~\gamma(\mathfrak{C}_1,N,\alpha)r^{\alpha}\mathcal{E}_{\mathbb{R}^N}^\mu(f,f).
\end{align}
For the second term on the right-hand side of (\ref{est5-1}), it is easy to verify that
\begin{equation*}
\|f_r\|_{L^2(B_\rho)}^2\leqslant\|f_r\|_{L^1(B_\rho)}\|f_r\|_{L^\infty(B_\rho)}.
\end{equation*}
Let us estimate the right-hand side separately. First, since $\text{supp}\,f\subset B_\rho$, we have
\begin{equation*}
\|f_r\|_{L^\infty(B_\rho)}=\sup_{x\in B_\rho}\bigg|\frac{1}{|B_r|}\int_{B_r(x)}f(y)\,dy\bigg|\leqslant\frac{1}{|B_r|}\int_{B_\rho}|f(y)|\,dy.
\end{equation*}
Next, using a similar covering argument as (\ref{est5-2}), we have
\begin{align*}
\int_{B_\rho}\bigg|\frac{1}{|B_r|}\int_{B_r(x)}f(y)\,dy\bigg|dx\,\leqslant&~\sum_{n=1}^\infty\int_{B_r(x_n)}\frac{1}{|B_r|}\int_{B_r(x)}|f(y)|\,dydx\\
\leqslant&~\sum_{n=1}^\infty\int_{B_r(x_n)}\frac{1}{|B_r|}\int_{B_{2r}(x_n)}|f(y)|\,dydx\\
\leqslant&~\gamma(N)\int_{B_\rho}|f(y)|\,dy.
\end{align*}
Therefore, we acquire
\begin{equation}\label{est5-3}
\|f_r\|_{L^2(B_\rho)}^2\leqslant\gamma(N)|B_r|^{-1}\|f\|_{L^1(B_\rho)}^2.
\end{equation}
Inserting (\ref{est5-2}) and (\ref{est5-3}) into (\ref{est5-1}), we get
\begin{equation*}
\|f\|_{L^2(B_\rho)}^2\leqslant\gamma(\mathfrak{C}_1,N,\alpha)\left(r^{\alpha}\mathcal{E}_{\mathbb{R}^N}^\mu(f,f)+|B_r|^{-1}\|f\|_{L^1(B_\rho)}^2\right)
\end{equation*}
for every $r>0$. Next, choose $r$ such that
\begin{equation*}
r^{\alpha}\mathcal{E}_{\mathbb{R}^N}^\mu(f,f)=|B_r|^{-1}\|f\|_{L^1(B_\rho)}^2
\end{equation*}
i.e.
\begin{equation*}
r^{N+\alpha}=\frac{\|f\|_{L^1(B_\rho)}^2}{|B_1|\mathcal{E}_{\mathbb{R}^N}^\mu(f,f)}.
\end{equation*}
As a result, we obtain that
\begin{equation*}
\|f\|_{L^2(B_\rho)}^2\leqslant\gamma(\mathfrak{C}_1,N,\alpha)\|f\|_{L^1(B_\rho)}^{\frac{2\alpha}{N+\alpha}}\mathcal{E}_{\mathbb{R}^N}^\mu(f,f)^{\frac{N}{N+\alpha}},
\end{equation*}
which we can also write as
\begin{equation}\label{est5-4}
\|f\|_{L^2(B_\rho)}^{2(1+\frac{\alpha}{N})}\leqslant\gamma(\mathfrak{C}_1,N,\alpha)\|f\|_{L^1(B_\rho)}^{\frac{2\alpha}{N}}\mathcal{E}_{\mathbb{R}^N}^\mu(f,f).
\end{equation}
\par
Now, let $v\in H_{\text{loc}}^\mu(B_\rho)$ satisfy $\text{supp}\,v\subset\Omega_0\subset B_\rho$. Using H\"older's inequality $\|v\|_{L^1(\Omega_0)}\leqslant|\Omega_0|^{\frac{1}{2}}\|v\|_{L^2(\Omega_0)}$, we apply (\ref{est5-4}) to $v$ and obtain that
\begin{align*}
\|v\|_{L^2(\Omega_0)}^{2(1+\frac{\alpha}{N})}=\|v\|_{L^2(B_\rho)}^{2(1+\frac{\alpha}{N})}\leqslant&~\gamma\|v\|_{L^1(B_\rho)}^{\frac{2\alpha}{N}}\mathcal{E}_{\mathbb{R}^N}^\mu(v,v)\\
=&~\gamma\|v\|_{L^1(\Omega_0)}^{\frac{2\alpha}{N}}\mathcal{E}_{\mathbb{R}^N}^\mu(v,v)\\
\leqslant&~\gamma|\Omega_0|^{\frac{\alpha}{N}}\|v\|_{L^2(\Omega_0)}^{\frac{2\alpha}{N}}\mathcal{E}_{\mathbb{R}^N}^\mu(v,v)
\end{align*}
for some $\gamma=\gamma(\mathfrak{C}_1,N,\alpha)$. After simplification, we get
\begin{equation*}
\|v\|_{L^2(\Omega_0)}^2\leqslant\gamma(\mathfrak{C}_1,N,\alpha)|\Omega_0|^{\frac{\alpha}{N}}\mathcal{E}_{\mathbb{R}^N}^\mu(v,v)
\end{equation*}
which is (\ref{FK}) for $c_0=\gamma(\mathfrak{C}_1,N,\alpha)$. Since $v$ was arbitrary, the desired result follows.
\end{proof}
\begin{remark}\upshape
It is worthy to mention that a generalized Faber-Krahn inequality, which was introduced in~\cite{chen-kuma1,chen-kuma2,kass5}, can play an equally effective role as (\ref{FK}) in the regularity analysis of integro-differential equations. Particularly, it can be stated as follows:\\
There exist $c_0,s>0$, such that for all $\Omega_0\subset B_\rho(x_0)$ and every $v\in H_{\text{loc}}^\mu(B_\rho(x_0))$ with $\text{supp}\,v\subset\Omega_0$,
\begin{equation*}
c_0\mathcal{E}_{\mathbb{R}^N}^{\mu}(v,v)\geqslant \rho^{-\alpha}\bigg(\frac{|B_\rho(x_0)|}{|\Omega_0|}\bigg)^s\int_{\Omega_0}|v|^2\,dx.
\end{equation*}
\end{remark}
In this paper, we often do not directly use (\ref{D}) when we estimate measures, instead we use the following corollary.
\begin{lemma}\label{D2}
If (\ref{D}) holds, then for every $0<r_1<r_2$,
\begin{equation*}
\sup_{x\in B_{r_1}(x_0)}\int_{\mathbb{R}^N\backslash B_{r_2}(x_0)}\mu(x,dy)\leqslant\frac{\mathfrak{C}_2}{(r_2-r_1)^\alpha}.
\end{equation*}
\end{lemma}
\begin{proof}
Pick $x\in B_{r_1}(x_0)$.
Since $r_1<r_2$, there has
$$\mathbb{R}^N\backslash B_{r_2}(x_0)\subset\mathbb{R}^N\backslash B_{r_2-r_1}(x).$$
Therefore, by using (\ref{D}), we have
\begin{equation*}
\int_{\mathbb{R}^N\backslash B_{r_2}(x_0)}\mu(x,dy)\leqslant\int_{\mathbb{R}^N\backslash B_{r_2-r_1}(x)}\mu(x,dy)\leqslant\frac{\mathfrak{C}_2}{(r_2-r_1)^\alpha}.
\end{equation*}
The proof is completed by the arbitrariness of $x$.
\end{proof}
The following lemma indicates that the assumption (\ref{D}) implies the local behavior of the measure $\mu$ under a weight of quadratic functions. The proof is inspired by \cite[Lemma~2.1]{kass6}.
\begin{lemma}\label{D->B}
If (\ref{D}) holds, then for every $r>0$, there holds
\begin{equation*}\label{B}
\frac{1}{r^2}\sup_{x\in\mathbb{R}^N}\int_{B_r(x)}|x-y|^2\mu(x,dy)\leqslant\gamma(\mathfrak{C}_2,\alpha)\frac{1}{r^\alpha}.
\end{equation*}
\begin{proof}
By using (\ref{D}), we have
\begin{align*}
\int_{B_r(x)}|x-y|^2\mu(x,dy)&=\sum_{n=0}^{\infty}\int_{B_{2^{-n}r}(x)\backslash B_{2^{-n-1}r}(x)}|x-y|^2\mu(x,dy)\\
&\leqslant\sum_{n=0}^{\infty}(2^{-n}r)^2\int_{B_{2^{-n}r}(x)\backslash B_{2^{-n-1}r}(x)}\mu(x,dy)\\
&\leqslant\mathfrak{C}_2\sum_{n=0}^{\infty}(2^{-n}r)^2(2^{-n-1}r)^{-\alpha}\\
&=\mathfrak{C}_2 r^{2-\alpha}\sum_{n=0}^{\infty}2^{-2n+\alpha n+\alpha}\\
&=\gamma(\mathfrak{C}_2,\alpha)r^{2-\alpha}.
\end{align*}
Since $x$ is arbitrary, the desired result follows.
\end{proof}
\end{lemma}
\par
Now we verify two properties of $\mu$ which satisfies (\ref{mu}) and (\ref{UJS}). The first property is stated as follows.
\begin{lemma}\label{ujs-control}
Let $\mu$ be represented by (\ref{mu}) and satisfy (\ref{UJS}). Then, for any non-negative function $v:\mathbb{R}^N\rightarrow\mathbb{R}$ and any $r>0,\rho\in(0,r/2)$, there holds
\begin{equation*}
\rho^N\sup_{x\in B_r}\int_{\mathbb{R}^N\backslash B_{r+\rho}}v(y)\,\mu(x,dy)\leqslant\gamma\int_{B_{r+\frac{\rho}{2}}}\int_{\mathbb{R}^N\backslash B_{r+\rho}}v(y)\,\mu(x,dy)dx,
\end{equation*}
where the constant $\gamma$ depends only on $N$ and $\hat{c}$.
\end{lemma}
\begin{proof}
Due to (\ref{mu}) and (\ref{UJS}), by Fubini's theorem we have
\begin{align*}
&\rho^N\sup_{x\in B_r}\int_{\mathbb{R}^N\backslash B_{r+\rho}}v(y)\,\mu(x,dy)\\
&~\qquad=\rho^N\sup_{x\in B_r}\int_{\mathbb{R}^N\backslash B_{r+\rho}}v(y)K(x,y)\,dy\\
&~\qquad\leqslant\gamma(N)2^N\hat{c}\sup_{x\in B_r}\int_{\mathbb{R}^N\backslash B_{r+\rho}}v(y)\bigg(\int_{B_{\rho/2}(x)}K(z,y)\,dz\bigg)dy\\
&~\qquad\leqslant\gamma\int_{\mathbb{R}^N\backslash B_{r+\rho}}v(y)\bigg(\int_{B_{r+\rho/2}}K(x,y)\,dx\bigg)dy\\
&~\qquad=\gamma\int_{B_{r+\rho/2}}\int_{\mathbb{R}^N\backslash B_{r+\rho}}v(y)K(x,y)\,dydx\\
&~\qquad=\gamma\int_{B_{r+\rho/2}}\int_{\mathbb{R}^N\backslash B_{r+\rho}}v(y)\,\mu(x,dy)dx
\end{align*}
for some $\gamma=\gamma(N,\hat{c})$.
\end{proof}
The second property is that a point-wise upper bound condition follows from (\ref{D}) and (\ref{UJS}).
\begin{lemma}\label{point-upper}
Let $\mu$ be represented by (\ref{mu}) and satisfy (\ref{D}), (\ref{UJS}). Then, there exists a constant $\gamma\geqslant 1$ depending only on $\mathfrak{C}_2,N,\alpha$ and $\hat{c}$, such that
\begin{equation*}
K(x,y)\leqslant \gamma\frac{1}{|x-y|^{N+\alpha}}\quad\text{for almost all}~~x,y\in\mathbb{R}^N.
\end{equation*}
\end{lemma}
\begin{proof}
For any $x,y\in\mathbb{R}^N$, set $r=\frac{|x-y|}{2}$. Obviously, it holds $B_r(x)\subset B_{3r}(y)$ and $r\leqslant|y-z|\leqslant 3r$ for every $z\in B_r(x)$. Therefore, by (\ref{D}), (\ref{UJS}) and the symmetry of $K$, we have
\begin{align*}
K(x,y)\leqslant&~\hat{c}\fint_{B_r(x)}K(z,y)dz\\
\leqslant&~\gamma r^{-N-2}\int_{B_r(x)}|y-z|^2K(z,y)dz\\
\leqslant&~\gamma r^{-N-2}\int_{B_{3r}(y)}|y-z|^2\mu(y,dz)\\
\leqslant&~\gamma |x-y|^{-N-\alpha},
\end{align*}
where $\gamma$ depends only on $\mathfrak{C}_2, N,\alpha$ and $\hat{c}$.
\end{proof}
As mentioned in the introduction, we provide an approach to the improved expansion of positivity lemma without applying the Krylov-Safanov covering argument. Instead, we pursue the approach of DiBenedetto, Gianazza and Vespri via a clustering lemma.
\par
In order to state our result, for all $x_0\in\mathbb{R}^N$ and all $\rho>0$ let
\[Q_\rho(x_0)=\prod_{i=1}^N\Big(x_0^i-\frac \rho 2,x_0^i+\frac \rho 2\Big)\]
be the $N$-dimensional open cube centered at $x_0$ with side length $\rho$.
\begin{lemma}[Clustering lemma]\label{cluster}
Let $\mu$ satisfy (\ref{A-poin}) and $M>0$. Assume $u\in H_{\mathrm{loc}}^\mu(Q_\rho(x_0))$ satisfies
\begin{itemize}
\item[(a)]$|\{u\geqslant M\}\cap Q_\rho(x_0)|>\beta|Q_\rho(x_0)|$ for some $\beta\in(0,1)$;
\item[(b)]$\displaystyle\int_{Q_\rho(x_0)}\int_{Q_\rho(x_0)}|u(x)-u(y)|^2\,\mu(x,dy)dx\leqslant \Lambda M^2\rho^{N-\alpha}$ for some $\Lambda>0$.
\end{itemize}
Then, for all $\delta,\lambda\in(0,1)$ there exist $x_1\in Q_\rho(x_0),\eta\in(0,1)$, such that
\begin{equation*}\label{clus}
|\{u>(1-\lambda) M\}\cap Q_{\eta\rho}(x_1)|>(1-\delta)|Q_{\eta\rho}(x_1)|,
\end{equation*}
where the constant $\eta$ depends only on $\mathfrak{C}_1,\alpha,\beta,\Lambda,\delta$ and $\lambda$.
\end{lemma}
\begin{proof}
Fix any $k\in\mathbb{N}$, we partition $Q_\rho(x_0)$ into a family $\mathcal{F}_k$ of cubes with side $\rho/k$, so that for all $Q\in\mathcal{F}_k$ we have
\begin{equation}\label{diag}
|Q|=\left(\frac{\rho}{k}\right)^N.
\end{equation}
We use $\#\mathcal{F}$ to represent the number of cubes in a family $\mathcal{F}$. Clearly, $\#\mathcal{F}_k=k^N$. Let $\beta\in(0,1)$ be as in (a) and divide all $Q\in\mathcal{F}_k$ into two classes according to
\begin{equation*}
\begin{cases}
Q\in\mathcal{F}_k^+\quad\text{if}\quad|Q\cap\{u>M\}|\geqslant\frac{\beta}{2}|Q|,\\
Q\in\mathcal{F}_k^-\quad\text{otherwise.}
\end{cases}
\end{equation*}
By hypothesis (a), for all $Q\in\mathcal{F}_k^+$ we have
\begin{align*}
\beta k^N|Q|=\beta|Q_\rho(x_0)|<&~|\{u\geqslant M\}\cap Q_\rho(x_0)|\\
=&\sum_{Q\in\mathcal{F}_k}|\{u\geqslant M\}\cap Q|\\
=&\sum_{Q\in\mathcal{F}_k^+}|\{u\geqslant M\}\cap Q|+\sum_{Q\in\mathcal{F}_k^-}|\{u\geqslant M\}\cap Q|.
\end{align*}
Dividing both sides by $|Q|$, by the definition of $\mathcal{F}_k^\pm$, we estimate
\begin{align*}
\beta k^N<&\sum_{Q\in\mathcal{F}_k^+}\frac{|\{u\geqslant M\}\cap Q|}{|Q|}+\sum_{Q\in\mathcal{F}_k^-}\frac{|\{u\geqslant M\}\cap Q|}{|Q|}\\
\leqslant&~\#\mathcal{F}_k^++\frac{\beta}{2}\#\mathcal{F}_k^-\\
=&~\frac{\beta}{2}k^N+\Big(1-\frac{\beta}{2}\Big)\#\mathcal{F}_k^+,
\end{align*}
which can be rephrased as
\begin{equation}\label{count}
\#\mathcal{F}_k^+>\frac{\beta}{2-\beta}k^N.
\end{equation}
Now fix $\delta,\lambda\in(0,1)$.
We claim that when $k$ is large enough, there must exist $Q\in\mathcal{F}_k^+$, such that
\begin{equation}\label{claim}
|\{u>(1-\lambda) M\}\cap Q|>(1-\delta)|Q|.
\end{equation}
Otherwise, for all $Q\in\mathcal{F}_k^+$ we have
\begin{equation}\label{est-m1}
|\{u\leqslant(1-\lambda) M\}\cap Q|\geqslant\delta|Q|.
\end{equation}
We will reach a contradiction. First, since $\lambda\in(0,1)$ and $Q\in\mathcal{F}_k^+$, we have
\begin{equation}\label{est-m2}
\Big|\Big\{u>\frac{2-\lambda}{2}M\Big\}\cap Q\Big|\geqslant|\{u>M\}\cap Q|\geqslant\frac{\beta}{2}|Q|.
\end{equation}
Next fix $x,y\in Q$, such that
\[u(x)\leqslant(1-\lambda) M,\quad u(y)>\frac{2-\lambda}{2}M,\]
which yields
\begin{equation}\label{est23}
u(y)-u(x)>\frac{\lambda}{2}M.
\end{equation}
Let $B$ be the ball circumscribing $Q$ such that $Q\subset B$. Applying (\ref{A-poin}) and H\"older's inequality, we have
\begin{align*}
\int_{Q}(u(x)-[u]_{Q})^2\,dx\leqslant&~2\int_Q(u(x)-[u]_B)^2dx+2\int_Q([u]_B-[u]_Q)^2\,dx\\
\leqslant&~4\int_B(u(x)-[u]_B)^2\,dx\\
\leqslant&~4\mathfrak{C}_1\left(\frac{\rho}{k}\right)^\alpha\int_B\int_B|u(x)-u(y)|^2\,\mu(x,dy)dx
\end{align*}
By using (\ref{est-m1}), (\ref{est-m2}) and (\ref{est23}), we estimate
\begin{align*}
\int_{Q}(u(x)-[u]_{Q})^2\,dx=&\int_Q\bigg(\fint_Q (u(y)-u(x))\,dy\bigg)^2dx\\
\geqslant&\int_{\{u\leqslant(1-\lambda)M\}\cap Q}\bigg(\frac{1}{|Q|}\int_{\{u>\frac{2-\lambda}{2}M\}\cap Q}(u(y)-u(x))\,dy\bigg)^2dx\\
\geqslant&~\frac{\beta^2\lambda^2M^2}{16}|\{u\leqslant(1-\lambda)M\}\cap Q|\\
\geqslant&~\frac{\beta^2\lambda^2M^2\delta}{16}|Q|.
\end{align*}
Therefore, under our current assumptions, the inequality
\begin{equation*}
\frac{\beta^2\lambda^2M^2\delta}{16}|Q|\leqslant 4\mathfrak{C}_1\left(\frac{\rho}{k}\right)^\alpha\int_B\int_B|u(x)-u(y)|^2\,\mu(x,dy)dx
\end{equation*}
holds for all $Q\in\mathcal{F}_k^+$. Summing over $Q\in\mathcal{F}_k^+$ and using (\ref{diag}) and (\ref{count}) we get
\begin{align*}
\frac{\beta^3\lambda^2M^2\delta}{16(2-\beta)}\rho^N\leqslant&~\frac{\beta}{2-\beta}k^N\frac{\beta^2\lambda^2M^2\delta}{16}|Q|\\
\leqslant&~4\mathfrak{C}_1\left(\frac{\rho}{k}\right)^\alpha\sum_{Q\in\mathcal{F}_k^+}\int_B\int_B|u(x)-u(y)|^2\,\mu(x,dy)dx\\
\leqslant&~\gamma(N,\mathfrak{C}_1)\left(\frac{\rho}{k}\right)^\alpha\int_{Q_\rho(x_0)}\int_{Q_\rho(x_0)}|u(x)-u(y)|^2\,\mu(x,dy)dx\\
\leqslant&~\frac{\Lambda\gamma(N,\mathfrak{C}_1) M^2\rho^N}{k^\alpha},
\end{align*}
where we used hypothesis (b) in the last line. Rewriting the above inequality we get
\begin{equation}\label{est-k}
k\leqslant\left(\frac{16(2-\beta)\Lambda\gamma(N,\mathfrak{C}_1)}{\beta^3\lambda^2\delta}\right)^{\frac{1}{\alpha}}.
\end{equation}
That means if all $Q\in\mathcal{F}_k^+$ violate (\ref{claim}), then $k$ must have an upper bound by (\ref{est-k}), which is a contradiction. Therefore, if we take
\begin{equation}\label{k}
k=\left(\frac{32(2-\beta)\Lambda\gamma(N,\mathfrak{C}_1)}{\beta^3\lambda^2\delta}\right)^{\frac{1}{\alpha}},
\end{equation}
there must exist $Q\in\mathcal{F}_k^+$, such that (\ref{claim}) holds. Then, let $x_1\in Q_\rho(x_0)$ be the center of $Q,\eta=\frac{1}{k}\in(0,1),$ and the conclusion follows.
\end{proof}
\begin{remark}\upshape
The idea behind the above clustering lemma mainly comes from \cite[Chapter~2,~Lemma~3.1]{dibe-gia-ves}. An adaptation to a fractional setting appears in \cite{duzgun}. However, the result of \cite{duzgun} does not apply to our general setting.
\end{remark}
\begin{remark}\upshape
This clustering lemma is stated on cubes, for balls, it is essentially equivalent. We only need to replace the segmentation argument in the above proof with a covering argument to see that.
\end{remark}
\section{Fundamental estimates}
This section contains several important modules for the proof of regularity estimates. Indeed, we deduce the energy estimates and the expansion of positivity lemma for the solutions of (\ref{equa}). In what follows, when we state ``$u$ is a sub(super)-solution..." and use ``$\pm$" or ``$\mp$", we mean the sub-solution corresponds to the upper sign and the super-solution corresponds to the lower sign in the statement.
\subsection{Energy estimates}
This subsection is devoted to energy estimates satisfied by local weak sub(super)-solutions of (\ref{equa}).
\begin{proposition}\label{DG-class}
Let $\mu$ satisfy (\ref{S}). Let $u$ be a local weak sub(super)-solution of (\ref{equa}) in $E$. Then, for every non-negative piecewise smooth cutoff function $\zeta(\cdot)$ compactly supported in $B_R\equiv B_R(x_0)\subset E$, there holds
\begin{eqnarray*}
&&\int_{B_R}\int_{B_R}(\zeta(y)w_\pm(y)-\zeta(x)w_\pm(x))^2\,\mu(x,dy)dx\nonumber\\
&&\qquad+\int_{B_R}w_{\pm}(x)\zeta^2(x)\,dx\bigg(\int_{\mathbb{R}^N}w_{\mp}(y)\,\mu(x,dy)\bigg)\nonumber\\
&&\quad\leqslant2\int_{B_R}\int_{B_R}\max\{w_{\pm}^2(x),w_{\pm}^2(y)\}(\zeta(x)-\zeta(y))^2\,\mu(x,dy)dx\nonumber\\
&&\qquad+2\int_{B_R}w_{\pm}(x)\zeta^2(x)\,dx\bigg(\underset{x\in\mathrm{supp}\,\zeta(\cdot)}{\mathrm{sup}}\int_{\mathbb{R}^N\backslash B_R}w_{\pm}(y)\,\mu(x,dy)\bigg).
\end{eqnarray*}
Here, we have denoted $w=u-k$.
\end{proposition}
\begin{proof}
We will only deal with the case of sub-solution as the other case is similar. We use $\varphi=w_{+}\zeta^2$ as a testing function in the weak form (\ref{def-solution}) of (\ref{equa}). Due to the support of $\zeta$ and (\ref{S}), $\mathcal{E}^{\mu}(u,\varphi)$ can be split into two parts, that is
\begin{align*}
\mathcal{E}^{\mu}(u,\varphi)=&\int_{\mathbb{R}^N}\int_{\mathbb{R}^N}(u(y)-u(x))[(u(y)-k)_{+}\zeta^2(y)-(u(x)-k)_{+}\zeta^2(x)]\,\mu(x,dy)dx\\
=&\int_{B_R}\int_{\mathbb{R}^N}(u(y)-u(x))[(u(y)-k)_{+}\zeta^2(y)-(u(x)-k)_{+}\zeta^2(x)]\,\mu(x,dy)dx\\
&+2\int_{B_R}\int_{\mathbb{R}^N\backslash B_R}(u(x)-u(y))(u(x)-k)_{+}\zeta^2(x)\,\mu(x,dy)dx\\
=:&~I_1+2I_2.
\end{align*}
In the first term, due to $ w_-(y) w_+(y)= w_-(x) w_+(x)=0$, we have
\begin{align*}\label{est35}
&(u(y)-u(x))[(u(y)-k)_{+}\zeta^2(y)-(u(x)-k)_{+}\zeta^2(x)]\nonumber\\
&\quad=[\zeta(y)w_+(y)-\zeta(x)w_+(x)]^2+w_-(y)w_+(x)\zeta^2(x)+w_-(x)w_+(y)\zeta^2(y)\nonumber\\
&~\qquad-w_+(y)w_+(x)(\zeta(x)-\zeta(y))^2.
\end{align*}
\par
Due to the support of $\zeta$ and symmetry of $\mu$, we have
\begin{align*}
I_1=&\int_{B_R}\int_{B_R}[\zeta(y)w_+(y)-\zeta(x)w_+(x)]^2\,\mu(x,dy)dx\\
&~+2\int_{B_R}\int_{\mathbb{R}^N}w_-(y)w_+(x)\zeta^2(x)\,\mu(x,dy)dx\\
&~-\int_{B_R}\int_{B_R}w_+(y)w_+(x)(\zeta(x)-\zeta(y))^2\,\mu(x,dy)dx\\
\geqslant&\int_{B_R}\int_{B_R}[\zeta(y)w_+(y)-\zeta(x)w_+(x)]^2\,\mu(x,dy)dx\\
&~+2\int_{B_R}\int_{\mathbb{R}^N}w_-(y)w_+(x)\zeta^2(x)\,\mu(x,dy)dx\\
&~-\int_{B_R}\int_{B_R}\max\{w_{+}^2(x),w_{+}^2(y)\}(\zeta(x)-\zeta(y))^2\,\mu(x,dy)dx.
\end{align*}
\par
To treat the second term, we first estimate
\begin{align*}
-(u(x)-u(y))(u(x)-k)_+\zeta^2(x)=&~(u(y)-u(x))(u(x)-k)_+\zeta^2(x)\\
\leqslant&~(u(y)-k)_+(u(x)-k)_+\zeta^2(x)\\
=&~w_+(y) w_+(x)\zeta^2(x).
\end{align*}
Therefore,
\begin{align*}
-I_2\leqslant&\int_{B_R}\int_{\mathbb{R}^N\backslash B_R} w_+(y) w_+(x)\zeta^2(x)\,\mu(x,dy)dx\\
\leqslant&\int_{B_R} w_+(x)\zeta^2(x)\,dx\bigg(\underset{x\in\text{supp}\,\zeta}{\text{sup}}\int_{\mathbb{R}^N\backslash B_R} w_+(y)\,\mu(x,dy)\bigg).
\end{align*}
Finally, we put all these estimates together and use the symmetry of $\mu$ to conclude.
\end{proof}
\subsection{Expansion of Positivity}
In regularity theory for elliptic and parabolic integro-differential equations, one of the essential tools is the expansion of positivity lemma. Before stating it, let us introduce two preliminary lemmas, i.e. a critical mass lemma and a measure shrinking lemma.
\begin{lemma}[critical mass]\label{Crit-Densi}
Let $\mu$ satisfy (\ref{S}), (\ref{A-poin}) and (\ref{D}). Let $B_R(x_0)\subset E$, $0<\rho<R$ and $M>0$. Let $u$ be a local weak super-solution of (\ref{equa}) in $E$, such that $u\geqslant 0$ in $B_R(x_0)$. There exists a constant $\nu\in(0,1)$ depending only on the {\it data} such that, if
\[|\{u<M\}\cap B_\rho(x_0)|\leqslant\nu|B_\rho(x_0)|,\]
then either
\[\left(\frac{\rho}{R}\right)^\alpha{\rm Tail}(u_-; x_0, \tfrac78R,R)>M\]
or
\[u\geqslant\frac{M}{2}\quad\text{a.e. in}\quad B_{\rho/2}(x_0).\]
\end{lemma}
\begin{proof}
Having fixed the pair of balls $B_{\rho/2}(x_0)\subset B_{\rho}(x_0),$ assume $x_0=0$ and introduce for $n\in\mathbb{N}_0,$
\begin{align*}
&k_n=\frac{M}{2}+\frac{M}{2^{n+1}},\quad\rho_n=\frac{\rho}{2}+\frac{\rho}{2^{n+1}},\quad\widetilde{\rho}_n=\frac{\rho_n+\rho_{n+1}}{2},\\
&B_n=B_{\rho_n},\qquad\widetilde{B}_n=B_{\widetilde{\rho}_n}.
\end{align*}
Observe that
$$B_{n+1}\subset\widetilde{B}_n\subset B_n.$$
Take $\zeta_n:\mathbb{R}^N\rightarrow\mathbb{R}$ to be a function such that $\text{supp}\,\zeta_n(\cdot)\subset \widetilde{B}_n,~\|\zeta_n\|_\infty\leqslant 1,~\zeta_n(x)=1$ in $B_{n+1}$ and $|\nabla\zeta|\leqslant\gamma(N)2^{n+3}\rho^{-1}$. According to assumption (\ref{D}) and Lemma~\ref{D->B}, we have
\begin{equation*}
\sup_{x\in B_{n}}\int_{B_{n}}(\zeta_n(y)-\zeta_n(x))^2\,\mu(x,dy)\leqslant \gamma(\mathfrak{C}_2,N,\alpha)\left(\frac{\rho}{2^{n+3}}\right)^{-\alpha}.
\end{equation*}
\par
By Proposition \ref{DG-class} and (\ref{S}), we have
\begin{align}\label{est1}
&\int_{B_{n}}\int_{B_{n}}|\zeta_n(y)(u(y)-k_n)_--\zeta_n(x)(u(x)-k_n)_-|^2\,\mu(x,dy)dx\nonumber\\
\leqslant&~\gamma(\mathfrak{C}_2,N,\alpha)\left(\frac{\rho}{2^{n+3}}\right)^{-\alpha}\int_{B_{n}}(u(x)-k_n)_-^2\,dx\nonumber\\
&~+2\int_{B_{n}}(u(x)-k_n)_-\,dx\bigg(\underset{x\in\text{supp}\,\zeta_n}{\text{sup}}\int_{\mathbb{R}^N\backslash B_n}(u(y)-k_n)_-\,\mu(x,dy)\bigg).
\end{align}
It is standard to obtain
\begin{equation}\label{est2}
\int_{B_n}(u(x)-k_n)_-^2dx\leqslant M^2|\{u<k_n\}\cap B_n|.
\end{equation}
\par
For the second term, recalling $u\geqslant 0$ in $B_R$ by assumption and $k_n\leqslant M$, we estimate
\begin{align*}
&\int_{B_{n}}(u(x)-k_n)_-\,dx\bigg(\underset{x\in\text{supp}\,\zeta_n}{\text{sup}}\int_{\mathbb{R}^N\backslash B_n}(u(y)-k_n)_-\,\mu(x,dy)\bigg)\\
\leqslant&~M|\{u<k_n\}\cap B_n|\bigg(\underset{x\in\widetilde{B}_n}{\text{sup}}\int_{\mathbb{R}^N\backslash B_n}k_n\,\mu(x,dy)+\underset{x\in\widetilde{B}_n}{\text{sup}}\int_{\mathbb{R}^N\backslash B_n}u_-(y)\,\mu(x,dy)\bigg)\\
\leqslant&~M|\{u<k_n\}\cap B_n|\bigg(M\mathfrak{C}_2 2^{(n+3)\alpha}\rho^{-\alpha}+\underset{x\in\widetilde{B}_0}{\text{sup}}\int_{\mathbb{R}^N\backslash B_n}u_-(y)\,\mu(x,dy)\bigg)\\
\leqslant&~M|\{u<k_n\}\cap B_n|\Big(M\mathfrak{C}_2 2^{(n+3)\alpha}\rho^{-\alpha}+\left(\frac{\rho}{R}\right)^\alpha{\rm Tail}(u_-;\tfrac78R,R)\Big),
\end{align*}
where we have used the assumption (\ref{D}) of $\mu$ and Lemma~\ref{D2}. If we impose
\[\left(\frac{\rho}{R}\right)^\alpha{\rm Tail}(u_-;\tfrac78R,R)\leqslant M,\]
then the above estimate becomes
\begin{flalign}\label{est3}
\int_{B_{n}}(u(x)-k_n)_-\,dx\bigg(\underset{x\in\text{supp}\,\zeta_n}{\text{sup}}\int_{\mathbb{R}^N\backslash B_n}(u(y)-k_n)_-\,\mu(x,dy)\bigg)\qquad\qquad\nonumber\\
\leqslant\gamma(\mathfrak{C}_2) 2^{(n+3)\alpha}M^2\rho^{-\alpha}|\{u<k_n\}\cap B_n|.
\end{flalign}
Combining (\ref{est1}), (\ref{est2}) and (\ref{est3}) we obtain
\begin{align*}
&\int_{B_{n}}\int_{B_{n}}|\zeta_n(y)(u(y)-k_n)_--\zeta_n(x)(u(x)-k_n)_-|^2\,\mu(x,dy)dx\nonumber\\
&\quad\qquad\leqslant\gamma2^{(n+3)\alpha}M^2\rho^{-\alpha}|\{u<k_n\}\cap B_n|,
\end{align*}
for some $\gamma=\gamma(\mathfrak{C}_2,N,\alpha)$.
Furthermore, since
\begin{align*}
&\int_{B_n}\int_{\mathbb{R}^N\backslash B_n}|\zeta_n(x)(u(x)-k_n)_-|^2\,\mu(x,dy)dx\\
&\qquad\quad\leqslant\underset{x\in \widetilde{B}_n}{\text{ sup}}\int_{\mathbb{R}^N\backslash B_n}\mu(x,dy)\int_{B_n}(u(x)-k_n)_-^2\,dx\\
&\qquad\quad\leqslant\mathfrak{C}_2 2^{(n+3)\alpha}\rho^{-\alpha}M^2|\{u<k_n\}\cap B_n|,
\end{align*}
we have
\begin{align*}
&\mathcal{E}_{\mathbb{R}^N}^\mu(\zeta_n(u-k_n)_-,\zeta_n(u-k_n)_-)\\
&\qquad\leqslant\int_{B_{n}}\int_{B_{n}}|\zeta_n(y)(u(y)-k_n)_--\zeta_n(x)(u(x)-k_n)_-|^2\,\mu(x,dy)dx\\
&\qquad\quad+2\int_{B_n}\int_{\mathbb{R}^N\backslash B_n}|\zeta_n(x)(u(x)-k_n)_-|^2\,\mu(x,dy)dx\\
&\qquad\leqslant~\gamma2^{(n+3)\alpha}M^2\rho^{-\alpha}|\{u<k_n\}\cap B_n|,
\end{align*}
for some $\gamma=\gamma(\mathfrak{C}_2,N,\alpha)$.
\par
An application of Lemma~\ref{poin->FK} with $v=\zeta_n(u-k_n)_-,\Omega_0=\widetilde{B}_n\cap\{u<k_n\}$ gives that
\begin{equation*}
c_0\mathcal{E}_{\mathbb{R}^N}^\mu(\zeta_n(u-k_n)_-,\zeta_n(u-k_n)_-)\geqslant|\widetilde{B}_n\cap\{u<k_n\}|^{-\frac{\alpha}{N}}\int_{\widetilde{B}_n\cap\{u<k_n\}}(u-k_n)_-^2\,dx.
\end{equation*}
By the definitions of $k_n,B_n,\widetilde{B}_n$, obviously we have
\begin{equation*}
B_{n+1}\cap\{u<k_{n+1}\}\subset\widetilde{B}_n\cap\{u<k_n\}\subset B_n\cap\{u<k_n\},
\end{equation*}
hence we estimate
\begin{align*}
\int_{\widetilde{B}_n\cap\{u<k_n\}}(u-k_n)_-^2\,dx\geqslant&\,\int_{B_{n+1}\cap\{u<k_{n+1}\}}(u-k_n)_-^2\,dx\\
\geqslant&\,\int_{B_{n+1}\cap\{u<k_{n+1}\}}(k_n-k_{n+1})^2\,dx\\
=&\,\Big(\frac{M}{2^{n+2}}\Big)^2|B_{n+1}\cap\{u<k_{n+1}\}|,
\end{align*}
and
\begin{equation*}
|\widetilde{B}_n\cap\{u<k_n\}|^{-\frac{\alpha}{N}}\geqslant|B_n\cap\{u<k_n\}|^{-\frac{\alpha}{N}}.
\end{equation*}
Combing the above inequalities, we get
\[|B_n\cap\{u<k_n\}|^{-\frac{\alpha}{N}}|B_{n+1}\cap\{u<k_{n+1}\}|\leqslant\gamma 2^{(2+\alpha)n}\rho^{-\alpha}|B_n\cap\{u<k_{n}\}|.\]
\par
In terms of $Y_n:={|\{u<k_{n}\}\cap B_{n}|}/{|B_n|},$ this estimate leads to
\[Y_{n+1}\leqslant\gamma 2^{(2+\alpha)n}Y_n^{1+\frac{\alpha}{N}}\]
for some $\gamma=\gamma({\it data})$.
By the fast geometric convergence, cf.~\cite[Chapter~2 Lemma~5.1]{dibe-gia-ves}, if we require that
\[Y_0\leqslant\gamma^{-\frac{N}{\alpha}}2^{-\frac{N(2+\alpha)}{\alpha^2}}=:\nu\quad\Longleftrightarrow\quad\frac{|\{u<M\}\cap B_\rho|}{|B_\rho|}\leqslant\nu,\]
then $Y_n\to 0$ as $n\to\infty$, i.e. $|\{u<\frac{M}{2}\}\cap B_{\frac{\rho}{2}}|=0.$
\end{proof}
To continue, we introduce the following \textit{logarithmic estimate}, which plays the key role in the proof of the measure shrinking lemma. The idea of proof comes from \cite{kass3}. This is a simplified version of \cite[Lemma 4.4]{kass1}.
\begin{proposition}\label{log-esti}
Let $\mu$ satisfy (\ref{S}) and (\ref{D}). Let $B_R(x_0)\subset E$. Let $u$ be a local weak super-solution of (\ref{equa}) in $E$, such that $u\geqslant 0$ in $B_R(x_0)$. There exists $\gamma>1$ depending only on the \textit{data}\,$\{\mathfrak{C}_2,N,\alpha\}$, such that for any $0<\rho<R/2$ and any $h>0$, the following estimate holds
\begin{align*}\label{log}
&\int_{B_{\rho}(x_0)}\int_{B_{\rho}(x_0)}\left|\log\left(\frac{u(x)+h}{u(y)+h}\right)\right|^2\mu(x,dy)dx\nonumber\\
&\qquad\leqslant\gamma\rho^{N-\alpha}\bigg[h^{-1}\left(\frac{\rho}{R}\right)^\alpha{\rm Tail}(u_-;x_0,\tfrac78R,R)+1\bigg].
\end{align*}
\end{proposition}
\begin{proof}
For simplicity, we assume $x_0=0$. Let $h>0$ be a parameter and
take $\zeta:\mathbb{R}^N\rightarrow\mathbb{R}$ to be a function, such that $\text{supp}\,\zeta(\cdot)\subset B_{\frac{7}{4}\rho},~\|\zeta\|_\infty\leqslant 1,~\zeta(x)=1$ in $B_\rho$ and $|\nabla\zeta|\leqslant\gamma(N)\rho^{-1}$. According to assumption (\ref{D}) and Lemma~\ref{D->B}, we have
\begin{equation}\label{est5-5}
\sup_{x\in B_{2\rho}}\int_{B_{2\rho}}(\zeta(y)-\zeta(x))^2\,\mu(x,dy)\leqslant \gamma(\mathfrak{C}_2,N,\alpha)\rho^{-\alpha}.
\end{equation}
\par
We use in the weak formulation (\ref{def-solution}) of (\ref{equa}), the testing function $\varphi$ defined by $\varphi=\frac{\zeta^2}{u+h}.$
Note that since $u\geqslant 0$ in the support of $\zeta$, the test function is well-defined. Consequently, we get
\begin{align}\label{est16}
&\int_{B_{2\rho}}\int_{B_{2\rho}}(u(x)-u(y))\bigg(\frac{\zeta^2(y)}{u(y)+h}-\frac{\zeta^2(x)}{u(x)+h}\bigg)\,\mu(x,dy)dx\nonumber\\
&\qquad\quad\leqslant2\int_{B_{2\rho}}\int_{\mathbb{R}^N\backslash B_{2\rho}}(u(x)-u(y))\frac{\zeta^2(x)}{u(x)+h}\,\mu(x,dy)dx.
\end{align}
For convenience, we denote
\begin{align*}
&I_1=\int_{B_{2\rho}}\int_{B_{2\rho}}(u(x)-u(y))\bigg(\frac{\zeta^2(y)}{u(y)+h}-\frac{\zeta^2(x)}{u(x)+h}\bigg)\,\mu(x,dy)dx,\\
&I_2=\int_{B_{2\rho}}\int_{\mathbb{R}^N\backslash B_{2\rho}}(u(x)-u(y))\frac{\zeta^2(x)}{u(x)+h}\,\mu(x,dy)dx.
\end{align*}
Setting $$E:=E(x,y)=\frac{u(x)+h}{u(y)+h}\qquad\text{and}\qquad F:=F(x,y)=\frac{\zeta(x)}{\zeta(y)},$$ we obtain
\begin{align*}
I_1=&\int_{B_{2\rho}}\int_{B_{2\rho}}\zeta(x)\zeta(y)\bigg(\frac{E}{F}+\frac{F}{E}-F-\frac{1}{F}\bigg)\,\mu(x,dy)dx\\
=&\int_{B_{2\rho}}\int_{B_{2\rho}}\zeta(x)\zeta(y)\Bigg[\bigg(\frac{E}{F}+\frac{F}{E}-2\bigg)-\bigg(\sqrt{F}-\frac{1}{\sqrt{F}}\bigg)^2\Bigg]\,\mu(x,dy)dx.
\end{align*}
By Taylor's theorem, it holds that
$$e^a+e^{-a}-2=2\sum_1^\infty\frac{a^{2k}}{(2k)!}\geqslant a^2,\quad\forall\, a\in\mathbb{R}.$$
Let $b=e^a$, then
$$b+\frac{1}{b}-2\geqslant(\log b)^2\quad\forall\, b>0.$$
Therefore, setting $b=\frac{E}{F}$, we have
\begin{equation*}
\frac{E}{F}+\frac{F}{E}-2\geqslant\left(\log E-\log F\right)^{2}.
\end{equation*}
Furthermore, it is easy to verify
\begin{equation*}
\zeta(x)\zeta(y)\bigg(\sqrt{F}-\frac{1}{\sqrt{F}}\bigg)^2=(\zeta(x)-\zeta(y))^2.
\end{equation*}
Thus, noticing $\zeta=1$ in $B_\rho$, we have
\begin{align}\label{est17}
I_1\geqslant&\int_{B_{2\rho}}\int_{B_{2\rho}}\zeta(x)\zeta(y)\left(\log E-\log F\right)^{2}\,\mu(x,dy)dx\nonumber\\
&-\int_{B_{2\rho}}\int_{B_{2\rho}}(\zeta(x)-\zeta(y))^2\,\mu(x,dy)dx\nonumber\\
=&\int_{B_{2\rho}}\int_{B_{2\rho}}\zeta(x)\zeta(y)\left(\log \frac{u(x)+h}{\zeta(x)}-\log \frac{u(y)+h}{\zeta(y)}\right)^2\,\mu(x,dy)dx\nonumber\\
&-\int_{B_{2\rho}}\int_{B_{2\rho}}(\zeta(x)-\zeta(y))^2\,\mu(x,dy)dx\nonumber\\
\geqslant&\int_{B_\rho}\int_{B_\rho}\left|\log\left(\frac{u(x)+h}{u(y)+h}\right)\right|^2\,\mu(x,dy)dx-\int_{B_{2\rho}}\int_{B_{2\rho}}(\zeta(x)-\zeta(y))^2\,\mu(x,dy)dx.
\end{align}

\par
For the right-hand side of (\ref{est16}), namely $I_2$, we can proceed as follows. First of all, notice that $u(y)\geqslant 0$ when $y\in B_R,$ and hence
\begin{equation*}
\frac{u(x)-u(y)}{u(x)+h}\leqslant 1\qquad\forall~ x\in B_{2\rho},~y\in B_R.
\end{equation*}
Moreover, when $y\in\mathbb{R}^N\backslash B_R$,
\begin{equation*}
u(x)-u(y)\leqslant[u(x)+u_-(y)]\qquad\forall~x\in B_{2\rho}.
\end{equation*}
Therefore, by splitting the integral region and using the fact that $\frac{u(x)}{u(x)+h}<1$, we have
\begin{align}\label{est18}
I_2\leqslant&\int_{B_{2\rho}}\int_{B_R\backslash B_{2\rho}}\frac{u(x)-u(y)}{u(x)+h}\zeta^2(x)\,\mu(x,dy)dx\nonumber\\
&+\int_{B_{2\rho}}\int_{\mathbb{R}^N\backslash B_{R}}\frac{u(x)-u(y)}{u(x)+h}\zeta^2(x)\,\mu(x,dy)dx\nonumber\\
\leqslant& \int_{B_{2\rho}}\int_{B_R\backslash B_{2\rho}}\zeta^2(x)\mu(x,dy)dx+\int_{B_{2\rho}}\int_{\mathbb{R}^N\backslash B_{R}}\frac{u(x)\zeta^2(x)}{u(x)+h}\mu(x,dy)dx\nonumber\\
&+\int_{B_{2\rho}}\int_{\mathbb{R}^N\backslash B_{R}}\frac{u_-(y)\zeta^2(x)}{u(x)+h}\mu(x,dy)dx\nonumber\\
\leqslant&\int_{B_{2\rho}}\int_{\mathbb{R}^N\backslash B_{2\rho}}\zeta^2(x)\,\mu(x,dy)dx+h^{-1}\int_{B_{2\rho}}\int_{\mathbb{R}^N\backslash B_{R}}u_-(y)\zeta^2(x)\,\mu(x,dy)dx.
\end{align}
By (\ref{D}), Lemma~\ref{D2} and the fact that the support of $\zeta$ is included in $B_{\frac{7}{4}\rho}$, we have
\begin{align}\label{est19}
\int_{B_{2\rho}}\int_{\mathbb{R}^N\backslash B_{2\rho}}\zeta^2(x)\,\mu(x,dy)dx\leqslant&~\gamma(N)\rho^N\underset{x\in B_{\frac{7}{4}\rho}}{\text{sup}}\int_{\mathbb{R}^N\backslash B_{2\rho}}\mu(x,dy)\nonumber\\
\leqslant&~\gamma(\mathfrak{C}_2,N)\rho^{N-\alpha}
\end{align}
and
\begin{align}\label{est20}
\int_{B_{2\rho}}\int_{\mathbb{R}^N\backslash B_{R}}u_-(y)\zeta^2(x)\,\mu(x,dy)dx\leqslant&~\gamma(N)|B_{2\rho}|R^{-\alpha}{\rm Tail}(u_-;\tfrac74\rho,R)\nonumber\\
\leqslant&~\gamma(N)\rho^{N-\alpha}\left(\frac{\rho}{R}\right)^\alpha{\rm Tail}(u_-;\tfrac78R,R).
\end{align}
By combining (\ref{est18}) with (\ref{est19}) and (\ref{est20}), we obtain
\[I_2\leqslant \gamma\rho^{N-\alpha}\left[1+h^{-1}\left(\frac{\rho}{R}\right)^\alpha{\rm Tail}(u_-;\tfrac78R,R)\right],\]
which together with (\ref{est17}) in (\ref{est16}) yields
\begin{align*}
&\int_{B_{\rho}}\int_{B_{\rho}}\left|\log\left(\frac{u(x)+h}{u(y)+h}\right)\right|^2\mu(x,dy)dx\\
&~\quad\leqslant\int_{B_{2\rho}}\int_{B_{2\rho}}|\zeta(x)-\zeta(y)|^2\,\mu(x,dy)dx
+\gamma\rho^{N-\alpha}\left[1+h^{-1}\left(\frac{\rho}{R}\right)^\alpha{\rm Tail}(u_-;\tfrac78R,R)\right].
\end{align*}
\par
Finally, by (\ref{est5-5}), we get the following estimate
\[\int_{B_{2\rho}}\int_{B_{2\rho}}|\zeta(x)-\zeta(y)|^2\,\mu(x,dy)dx\leqslant \gamma(\mathfrak{C}_2,N,\alpha)\rho^{-\alpha}|B_{2\rho}|=\gamma(\mathfrak{C}_2,N,\alpha)\rho^{N-\alpha}\]
to conclude the proof.
\end{proof}

Now we are ready to deduce the measure shrinking lemma.
\begin{lemma}[measure shrinking]\label{Mea-Shri}
Let $\mu$ satisfy (\ref{S}), (\ref{A-poin}) and (\ref{D}). Let $B_R(x_0)\subset E,~\beta\in(0,1]$ and $M>0$. Let $u$ be a local weak super-solution of (\ref{equa}) in $E$, such that $u\geqslant 0$ in $B_R(x_0)$. There exists $\gamma>1$ depending only on the {\it data}, such that for any $0<\rho<R$, if
\[|\{u\geqslant M\}\cap B_\rho(x_0)|\geqslant\beta|B_\rho(x_0)|,\]
then for any $\sigma\in(0,1),$ either
\[\left(\frac{\rho}{R}\right)^\alpha{\rm Tail}(u_-;x_0,\tfrac78R,R)>\sigma M\]
or
\[|\{u\leqslant \sigma M\}\cap B_\rho(x_0)|\leqslant\frac{\gamma}{\beta}\bigg[\log\bigg(\frac{1+\sigma}{2\sigma}\bigg)\bigg]^{-1}|B_\rho(x_0)|.\]
\end{lemma}
\begin{proof}
For simplicity, we assume $x_0=0.$ As a restatement, we will show that there exists $\gamma$, such that if
\begin{equation}\label{condi1}
\left(\frac{\rho}{R}\right)^\alpha{\rm Tail}(u_-;\tfrac78R,R)\leqslant\sigma M
\end{equation}
and if
\begin{equation}\label{condi2}
|\{u\geqslant M\}\cap B_\rho|\geqslant\beta|B_\rho|,
\end{equation}
then
\begin{equation}\label{resul1}
|\{u\leqslant \sigma M\}\cap B_\rho|\leqslant\frac{\gamma}{\beta}\bigg[\log\bigg(\frac{1+\sigma}{2\sigma}\bigg)\bigg]^{-1}|B_\rho|.
\end{equation}
Consider the function $v$ defined as follows
\begin{equation}\label{v-def}
v:=\log_+\left(\frac{M+h}{u+h}\right),
\end{equation}
where $h>0$ is to be chosen.
By the assumption (\ref{A-poin}) of $\mu$ we get
\begin{equation*}
\int_{B_\rho}(v-[v]_{B_\rho})^2\,dx\leqslant \mathfrak{C}_1\rho^\alpha\int_{B_\rho}\int_{B_\rho}|v(x)-v(y)|^2\,\mu(x,dy)dx.
\end{equation*}
By the definition of $v$, it follows that
\[\int_{B_\rho}\int_{B_\rho}|v(x)-v(y)|^2\,\mu(x,dy)dx\leqslant\int_{B_{\rho}}\int_{B_{\rho}}\left|\log\left(\frac{u(x)+h}{u(y)+h}\right)\right|^2\,\mu(x,dy)dx.\]
Combining the previous two estimates and applying Proposition~\ref{log-esti}, we get
\[\fint_{B_\rho}|v-[v]_{B_\rho}|^2\,dx\leqslant \gamma\left(h^{-1}\left(\frac{\rho}{R}\right)^\alpha{\rm Tail}(u_-;\tfrac78R,R)+1\right)\]
for some $\gamma=\gamma(\mathfrak{C}_1,\mathfrak{C}_2,N)$.
Therefore, choosing $h=\sigma M$ and using (\ref{condi1}), we get from the last line that
\begin{equation}\label{est21}
\fint_{B_\rho}|v-[v]_{B_\rho}|^2\,dx\leqslant \gamma,
\end{equation}
where the constant $\gamma$ depends only on $\mathfrak{C}_1,\mathfrak{C}_2$ and $N$.
\par
Using the definition of $v$ given in (\ref{v-def}), we obtain
\begin{equation}\label{est43}
|\{v=0\}\cap B_\rho|\cdot[v]_{B_\rho}=\int_{\{v=0\}\cap B_\rho}|v-[v]_{B_\rho}|\,dx\leqslant\int_{B_\rho}|v-[v]_{B_\rho}|\,dx.
\end{equation}
By (\ref{condi2}), (\ref{est43}), H\"older's inequality and (\ref{est21}) we have
\begin{equation*}
\beta\fint_{B_\rho}v\,dx\leqslant\frac{|\{v=0\}\cap B_\rho|}{|B_\rho|}\fint_{B\rho}v\,dx\leqslant\fint_{B_\rho}|v-[v]_{B_\rho}|\,dx\leqslant\gamma\bigg(\fint_{B_\rho}|v-[v]_{B_\rho}|^2\,dx\bigg)^{\frac{1}{2}}\leqslant\gamma.
\end{equation*}
The left-hand side of the last inequality has the lower bound as follows,
\begin{equation*}
\beta\fint_{B_\rho}v\,dx\geqslant\frac{\beta}{|B_\rho|}\int_{\{u\leqslant\sigma M\}\cap B_\rho}\log_+\left(\frac{M+h}{u+h}\right)\,dx\geqslant\beta\frac{|\{u\leqslant\sigma M\}\cap B_\rho|}{|B_\rho|}\log\left(\frac{1+\sigma}{2\sigma}\right).
\end{equation*}
Combining the last two inequalities yields (\ref{resul1}). We have completed the proof.
\end{proof}
Combining the critical mass lemma and the measure shrinking lemma yields the expansion of positivity lemma as follows.
\begin{proposition}[expansion of positivity]\label{Exp-Pos}
Let $\mu$ satisfy (\ref{S}), (\ref{A-poin}) and (\ref{D}). Let $B_R(x_0)\subset E$ and $0<\rho<R/4$. Let $u$ be a local weak super-solution of (\ref{equa}), such that $u\geqslant 0$ in $B_R(x_0)$. Suppose for some constants $\beta\in(0,1]$ and $M>0,$ there holds
\[|\{u\geqslant M\}\cap B_\rho(x_0)|\geqslant\beta|B_\rho(x_0)|.\]
Then there exists a constant $\sigma\in(0,1)$ depending only on the \textit{data} and $\beta$, such that either
\[\left(\frac{\rho}{R}\right)^\alpha{\rm Tail}(u_-;x_0,\tfrac78R,R)>\sigma M\]
or
\[u\geqslant\sigma M\quad\text{a.e. in}\quad B_{2\rho}(x_0).\]
\end{proposition}
\begin{proof}
For simplicity, we assume $x_0=0.$ Since
\[|\{u\geqslant M\}\cap B_{4\rho}|\geqslant|\{u\geqslant M\}\cap B_\rho|\geqslant\beta|B_\rho|=4^{-N}\beta|B_{4\rho}|,\]
using Lemma~\ref{Mea-Shri}, if
\begin{equation}\label{est6-2}
4^\alpha\left(\frac{\rho}{R}\right)^\alpha{\rm Tail}(u_-;\tfrac78R,R)\leqslant\sigma M,
\end{equation}
we have
\begin{equation*}
|\{u\leqslant \sigma M\}\cap B_{4\rho}|\leqslant4^N\frac{\gamma}{\beta}\bigg[\log\bigg(\frac{1+\sigma}{2\sigma}\bigg)\bigg]^{-1}|B_{4\rho}|,
\end{equation*}
where $\gamma=\gamma({\it data})>1$ is the constant from Lemma~\ref{Mea-Shri}. Take $\sigma$ small enough to satisfy
$$4^N\frac{\gamma}{\beta}\left[\log\left(\frac{1+\sigma}{2\sigma}\right)\right]^{-1}\leqslant\nu,$$
where $\nu=\nu({\it data})\in(0,1)$ is the constant from Lemma~\ref{Crit-Densi}. Using Lemma~\ref{Crit-Densi}, we have if (\ref{est6-2}) holds, then
\[u\geqslant\frac{\sigma M}{2}\quad\text{a.e. in}~B_{2\rho}.\]
Redefining $\min\{\frac{\sigma}{4^\alpha},\frac{\sigma}{2}\}$ as $\sigma$ finishes the proof.
\end{proof}
\section{Local H\"older continuity}
This section is devoted to the proof of Theorem~\ref{Holder}.
\par
Without loss of generality, we take $x_0=0$ and set
\begin{equation*}
\mu^+=\underset{B_R}{\text{sup}}~u,\quad \mu^-=\underset{B_R}{\text{inf}}~u.
\end{equation*}
Obviously, $\underset{B_R}{\text{osc}}~u=\mu^+-\mu^-\leqslant\omega$, where
\begin{equation*}
\omega=2~\underset{B_R}{\rm{sup}}~|u|+{\rm Tail}(|u|;\tfrac78R,R).
\end{equation*}
From now on, let $\sigma=\sigma({\it data})\in (0,1)$ be determined in Proposition~\ref{Exp-Pos} with $\beta=\frac{1}{2}.$ For some $c\in(0,\frac{1}{4})$ to be chosen, consider two alternatives
\begin{equation*}
\begin{cases}
|\{u(\cdot)-\mu^-\geqslant\frac{1}{4}\omega\}\cap B_{cR}|\geqslant\frac{1}{2}|B_{cR}|,\\
|\{\mu^+-u(\cdot)\geqslant\frac{1}{4}\omega\}\cap B_{cR}|\geqslant\frac{1}{2}|B_{cR}|.
\end{cases}
\end{equation*}
Assuming $\mu^+-\mu^-\geqslant\frac{1}{2}\omega,$ one of the two alternatives must hold. Let us suppose the first alternative holds for instance.
Proposition~\ref{Exp-Pos} is applied to $u-\mu^-$ instead of $u$ with $\beta=\frac{1}{2},~M=\frac{1}{4}\omega,~\rho=cR$. Then we obtain that either
\begin{equation}\label{est6}
c^\alpha{\rm Tail}((u-\mu^-)_-;\tfrac78R,R)>\frac{1}{4}\sigma\omega
\end{equation}
or
\[u-\mu^-\geqslant\frac{\sigma}{4}\omega\quad\text{a.e. in}~B_{cR},\]
which gives the reduction of oscillation
\begin{equation}\label{est7}
\underset{B_{cR}}{\text{osc}}~u\leqslant\mu^+-\mu^--\frac{1}{4}\sigma\omega\leqslant\Big(1-\frac{1}{4}\sigma\Big)\omega=:\omega_1.
\end{equation}
The number $c$ is chosen to ensure that (\ref{est6}) does not happen. Indeed, we may first estimate by using (\ref{D}), Lemma~\ref{D2} and the definition of $\omega$,
\begin{equation*}
{\rm Tail}((u-\mu^-)_-;\tfrac78R,R)
\leqslant\mathfrak{C}_2|\mu^-| R^\alpha(R/8)^{-\alpha}+{\rm Tail}(u_-;\tfrac78R,R)\leqslant\gamma(\mathfrak{C}_2)\omega.
\end{equation*}
Then, we choose
\[c^\alpha \gamma(\mathfrak{C}_2)\omega\leqslant\frac{1}{4}\sigma\omega\quad\text{i.e.}\quad c\leqslant\Big(\frac{\sigma}{4\gamma}\Big)^{1/\alpha},\]
such that (\ref{est6}) does not occur. If $\mu^+-\mu^-<\omega/2,$ there holds
\[\underset{B_{cR}}{\text{osc}}~u<\underset{B_{R}}{\text{osc}}~u<\frac{\omega}{2}<\Big(1-\frac{1}{4}\sigma\Big)\omega\]
because of $c\in(0,1/4)$ and $\sigma\in(0,1).$ That means (\ref{est7}) still holds when $\mu^+-\mu^-<\frac12\omega.$
\par
Now, we may proceed by induction. Suppose up to $i=1,\dots,j$, we have built
\begin{equation*}
\begin{cases}
r_0=R,~~~r_i=cr_{i-1},~~~B_i=B_{r_i},\\
\omega_0=\omega,~~\omega_i=(1-\frac \sigma 8)~\omega_{i-1},\\
\mu_i^+=\underset{B_i}{\text{sup}}~u,~\mu_i^-=\underset{B_i}{\text{inf}}~u,~\underset{B_i}{\text{osc}}~u\leqslant\omega_i.
\end{cases}
\end{equation*}
The induction argument will show that the above oscillation estimate continues to hold for the $(j+1)$-th step. Let $\sigma$ be fixed as before, whereas $c\in(0,1/4)$ is subject to a further choice. We consider two alternatives
\begin{equation*}
\begin{cases}
|\{u(\cdot)-\mu_j^-\geqslant\frac{1}{4}\omega_j\}\cap B_{cr_j}|\geqslant\frac{1}{2}|B_{cr_j}|,\\
|\{\mu_j^+-u(\cdot)\geqslant\frac{1}{4}\omega_j\}\cap B_{cr_j}|\geqslant\frac{1}{2}|B_{cr_j}|.
\end{cases}
\end{equation*}
Like in the first step, we may assume $\mu_j^+-\mu_j^-\geqslant\frac{1}{2}\omega_j,$ so that one of the two alternatives must hold. Otherwise, the case $\mu_j^+-\mu_j^-<\frac{1}{2}\omega_j$ can be trivially incorporated into the forthcoming oscillation estimate (\ref{est9}).
\par
Let us suppose the first case holds for instance. Proposition~\ref{Exp-Pos} is applied in $B_j$ to $u-\mu_j^-$ intead of $u$ with $\beta=\frac{1}{2},~M=\frac{1}{4}\omega_j,~\rho=cr_j$. Then we obtain that either
\begin{equation}\label{est8}
c^\alpha{\rm Tail}((u-\mu_j^-)_-;\tfrac78r_j,r_j)>\frac{1}{4}\sigma\omega_j
\end{equation}
or
\begin{equation*}
u-\mu_j^-\geqslant\frac{1}{4}\sigma\omega_j \quad\text{a.e. in}~B_{cr_j},
\end{equation*}
which, thanks to the $j$-th induction assumption, gives the reduction of oscillation
\begin{equation}\label{est9}
\underset{B_{cr_j}}{\text{osc}}u\leqslant\mu_j^+-\mu_j^--\frac{1}{4}\sigma\omega_j\leqslant\Big(1-\frac{1}{4}\sigma\Big)\omega_j=:\omega_{j+1}.
\end{equation}
The final choice of $c$ is made to ensure that (\ref{est8}) does not happen, independent of $j$. In order for that, we first rewrite the tail as follows:
\begin{align*}
&{\rm Tail}((u-\mu_j^-)_-;\tfrac78r_j,r_j)\\
&~\quad=r_j^\alpha\underset{x\in B_{\frac{7}{8}r_j}}{\text{sup}}\bigg[\int_{\mathbb{R}^N\backslash B_{R}}(u(y)-\mu_j^-)_-\,\mu(x,dy)+\sum_{i=1}^{j}\int_{B_{i-1}\backslash B_i}(u(y)-\mu_j^-)_-\,\mu(x,dy)\bigg].
\end{align*}
The first integral is estimated by using (\ref{D}), Lemma~\ref{D2} and the definition of $\omega$. Indeed, noticing that $r_j\leqslant R$, we have
\begin{eqnarray*}
\underset{x\in B_{\frac{7}{8}r_j}}{\text{sup}}\int_{\mathbb{R}^N\backslash B_{R}}(u(y)-\mu_j^-)_-\,\mu(x,dy)\leqslant\frac{8^{\alpha}\mathfrak{C}_2}{R^{\alpha}}|\mu_j^-|+\frac{1}{R^{\alpha}}{\rm Tail}(u_-;\tfrac78R,R)\leqslant\gamma\frac{\omega} {R^{\alpha}},
\end{eqnarray*}
where $\gamma$ depends only on $\mathfrak{C}_2$.
Whereas the second integral is estimated by using the simple fact that, for $i=1,\dots,j,$
\[(u-\mu_j^-)_-\leqslant\mu_j^--\mu_{i-1}^-\leqslant\mu_j^+-\mu_{i-1}^-\leqslant\mu_{i-1}^+-\mu_{i-1}^-\leqslant\omega_{i-1}~\text{a.e. in}~B_{i-1}.\]
Using also the assumption (\ref{D}) of $\mu$ and Lemma~\ref{D2}, we have
\begin{align*}
\underset{x\in B_{\frac{7}{8}r_j}}{\text{sup}}\int_{B_{i-1}\backslash B_i}(u(y)-\mu_j^-)_-\,\mu(x,dy)\leqslant&~\underset{x\in B_{\frac{7}{8}r_i}}{\text{sup}}\int_{B_{i-1}\backslash B_i}\omega_{i-1}\,\mu(x,dy)\\
\leqslant&~\mathfrak{C}_2\omega_{i-1}\left(\frac{r_i}{8}\right)^{-\alpha}\leqslant\gamma\frac{\omega_{i-1}}{r_i^{\alpha}},
\end{align*}
where $\gamma$ depends only on $\mathfrak{C}_2$. Combine them and further estimate the tail by
\begin{align*}
{\rm Tail}((u-\mu_j^-)_-;\tfrac78r_j,r_j)
\leqslant&~\gamma\bigg(r_j^\alpha\frac{\omega}{R^\alpha}+r_j^\alpha\sum_{i=1}^{j}\frac{\omega_{i-1}}{r_i^\alpha}\bigg)\\
\leqslant&~\gamma\omega_j\bigg[\Big(1-\frac{1}{4}\sigma\Big)^{-j}c^{j\alpha}+\sum_{i=1}^{j}\Big(1-\frac{1}{4}\sigma\Big)^{i-j-1}c^{(j-i)\alpha}\bigg]\\
\leqslant&~\gamma\omega_j\bigg[\sum_{i=0}^{j}\Big(1-\frac{1}{4}\sigma\Big)^{i-j-1}c^{(j-i)\alpha}\bigg].
\end{align*}
The summation in the last display is bounded by $2(1-\frac14\sigma)^{-1}$ if we restrict the choice of $c$ by
\[c^\alpha\Big(1-\frac{1}{4}\sigma\Big)^{-1}\leqslant\frac{1}{2}\quad\text{i.e.}\quad c\leqslant 2^{-\frac{1}{\alpha}}\Big(1-\frac{1}{4}\sigma\Big)^{\frac{1}{\alpha}}.\]
Consequently, if we choose $c$ to verify
\[c^\alpha\gamma\Big(1-\frac{1}{4}\sigma\Big)^{-1}\leqslant\frac{\sigma}{4}\quad\text{i.e.}\quad c\leqslant\Big[\frac{\sigma}{4\gamma}\Big(1-\frac{1}{4}\sigma\Big)\Big]^{\frac{1}{\alpha}},\]
where $\gamma$ is from the last tail estimate, then (\ref{est8}) does not happen. The final choice of $c$ is as follows
\[c:=\min\bigg\{\frac{1}{8},\Big[\frac{\sigma}{4\gamma}\Big(1-\frac{1}{4}\sigma\Big)\Big]^{\frac{1}{\alpha}},2^{-\frac{1}{\alpha}}\Big(1-\frac{1}{4}\sigma\Big)^{\frac{1}{\alpha}}\bigg\}<\frac{1}{4}.\]
Letting $r_{j+1}=cr_j$, we obtain from (\ref{est9}) that
\[\underset{B_{j+1}}{\text{osc}}~u\leqslant\omega_{j+1},\]
which completes the induction argument. Now, we have
\[\underset{B_{c^iR}}{\text{osc}}~u\leqslant\Big(1-\frac{1}{4}\sigma\Big)^i\omega,~~i=0,1,2,\dots\]
Let $\rho\in(0,R)$. There must be some $i\in\mathbb{N}_0$ such that
\[c^{i+1}R\leqslant\rho< c^iR.\]
As a result, we have
\[\underset{B_{\rho}}{\text{osc}}~u\leqslant\underset{B_{c^iR}}{\text{osc}}~u\leqslant\Big(1-\frac{1}{4}\sigma\Big)^i\omega\]
and
\[i+1\leqslant\frac{\ln(\rho/R)}{\ln c},\]
which means
\[\underset{B_{\rho}}{\text{osc}}~u\leqslant\frac{4}{4-\sigma}\left(\frac{\rho}{R}\right)^{\frac{\ln(1-\sigma/4)}{\ln c}}\omega.\]
Setting $\alpha_*=\frac{\ln(1-\sigma/4)}{\ln c}$, we complete the proof.
\section{Weak Harnack inequalities}\label{weak-Harn}
In this section, we obtain a Weak Harnack inequality for local weak super-solution of (\ref{equa}), namely Theorem~\ref{weak-harn}. Before proving the main result, let us first introduce the following useful proposition.
\begin{proposition}[improved expansion of positivity]\label{Exp-Pos-im}
Let $\mu$ satisfy (\ref{S}), (\ref{A-poin}) and (\ref{D}). Let $B_R(x_0)\subset E,0<\rho<R/32,\beta\in(0,1)$ and $M>0$. Let $u$ be a local weak super-solution of (\ref{equa}) in $E$, such that $u\geqslant 0$ in $B_R(x_0)$. Suppose
\[|\{u\geqslant M\}\cap B_\rho(x_0)|\geqslant\beta|B_\rho(x_0)|.\]
Then, there exist constants $\xi\in(0,1)$ and $q>1$ depending only on {\it data}, such that either
\[\left(\frac{\rho}{R}\right)^\alpha{\rm Tail}(u_-;x_0,\tfrac78R,R)>\xi\beta^qM\]
or
\[u\geqslant{\xi}\beta^qM\quad\text{a.e. in}~B_{2\rho}(x_0).\]
\end{proposition}
\begin{remark}\upshape
The weak Harnack inequality indicates apostiori that $\sigma$ should have a power-like dependence on $\beta$ in Proposition~\ref{Exp-Pos}. However, as one can easily check, the dependence on $\beta$ of $\sigma$ in Proposition~\ref{Exp-Pos} has the form $\sigma\approx e^{-\gamma/\beta}$. Proposition~\ref{Exp-Pos-im} improves this dependence to a power-type, with the help of the previous clustering lemma (see Lemma~\ref{cluster}).
\end{remark}
\begin{proof}[Proof of the improved Expansion of Positivity]
For simplicity, we assume $x_0=0.$ As a restatement, we need to show that there exist $\xi\in(0,1),~q>1$ such that if
\begin{equation}\label{est-tailcontrol}
\left(\frac{\rho}{R}\right)^\alpha{\rm Tail}(u_-;\tfrac78R,R)\leqslant\xi\beta^qM,
\end{equation}
then
\[u\geqslant{\xi}\beta^qM\quad\text{a.e. in}~B_{2\rho}.\]
We want to apply the clustering lemma. For this, fix $\lambda=\delta=1/2$.
Let
\[u_M:=\min\{u,M\}=M-(u-M)_-.\]
It is helpful to observe that
\[\{u_M\geqslant\widetilde\lambda M\}=\{u\geqslant\widetilde\lambda M\}\quad\text{for any}\quad\widetilde\lambda\in(0,1].\]
From this and the condition in Proposition~\ref{Exp-Pos-im} we know that hypothesis (a) of Lemma~\ref{cluster} is satisfied for $u_M$.
Take $\zeta:\mathbb{R}^N\rightarrow\mathbb{R}$ to be a function such that $\text{supp}\,\zeta(\cdot)\subset B_{3\rho/2},~\|\zeta\|_\infty\leqslant 1,~\zeta(x)=1$ in $B_{\rho}$ and $|\nabla \zeta|\leqslant\gamma(N)\rho^{-1}$. According to assumption (\ref{D}) and Lemma~\ref{D->B}, we have
\begin{equation*}
\sup_{x\in B_{\rho}}\int_{B_{\rho}}(\zeta(y)-\zeta(x))^2\,\mu(x,dy)\leqslant\gamma(\mathfrak{C}_2,N,\alpha)\rho^{-\alpha}.
\end{equation*}
The following analysis is analogous to the proof of Lemma~\ref{Crit-Densi}.
\par
Since $u$ is a super-solution to (\ref{equa}) in $E$, by Proposition~\ref{DG-class} and (\ref{S}) we have
\begin{align*}
&\int_{B_\rho}\int_{B_\rho}|u_M(x)-u_M(y)|^2\,\mu(x,dy)dx\\
&~\quad=\int_{B_\rho}\int_{B_\rho}|(u(y)-M)_--(u(x)-M)_-|^2\,\mu(x,dy)dx\\
&~\quad\leqslant\gamma(\mathfrak{C}_2,N,\alpha)\rho^{-\alpha}\int_{B_{2\rho}}(u(x)-M)_-^2\,dx\\
&~\qquad+\gamma(2\rho)^{-\alpha}{\rm Tail}((u-M)_-;\tfrac32\rho,2\rho)\int_{B_{2\rho}}(u(x)-M)_-\,dx.
\end{align*}
Recalling that $u\geqslant 0$ in $B_R$ and $B_{4\rho}\subset B_R$, we have
\[\int_{B_{2\rho}}(u(x)-M)_-^2\,dx\leqslant M^2|B_{2\rho}|,\]
and
\begin{align*}
&(2\rho)^{-\alpha}{\rm Tail}((u-M)_-;\tfrac32\rho,2\rho)\int_{B_{2\rho}}(u(x)-M)_-\,dx\\
&~\qquad\leqslant M|B_{2\rho}|\left(2^\alpha\mathfrak{C}_2 M\rho^{-\alpha}+R^{-\alpha}{\rm Tail}(u_-;\tfrac38R,R)\right)\\
&~\qquad\leqslant M|B_{2\rho}|(2^\alpha\mathfrak{C}_2 M\rho^{-\alpha}+M\rho^{-\alpha})\\
&~\qquad\leqslant\gamma(\mathfrak{C}_2,\alpha)M^2\rho^{-\alpha}|B_{2\rho}|,
\end{align*}
where we used (\ref{D}), Lemma~\ref{D2}, (\ref{est-tailcontrol}) and $\xi,\beta\in(0,1)$. Thus,
\[\int_{B_\rho}\int_{B_\rho}|u_M(x)-u_M(y)|^2\,\mu(x,dy)dx\leqslant\gamma M^2\rho^{N-\alpha}\]
for some $\gamma=\gamma(\mathfrak{C}_2,N,\alpha)$. From this we know that hypothesis (b) of Lemma~\ref{cluster} is satisfied with $c_0=\gamma(\mathfrak{C}_2,N,\alpha)$. Set $\eta$ to be the constant from Lemma~\ref{cluster}, depending only on {\it data} and $\beta$. Recalling (\ref{k}) in the proof of Lemma~\ref{cluster}, we know that
\begin{equation}\label{est44}
\eta=\widetilde\eta\beta^{\frac3\alpha}
\end{equation}
for some $\widetilde\eta=\widetilde\eta({\it data})$.
\par
We apply Lemma~\ref{cluster} to $u_M$ with $\lambda=\delta=1/2$ and obtain that there exists $x_1\in B_\rho$ such that
\[|\{u>M/2\}\cap B_{\eta\rho}(x_1)|=|\{u_M>M/2\}\cap B_{\eta\rho}(x_1)|\geqslant\frac{1}{2}|B_{\eta\rho}(x_1)|.\]
Set $\sigma_{0},\sigma_1$ to be the constant from Proposition~\ref{Exp-Pos} with $\beta=1/2$ and $\beta=1$ respectively, which depend on {\it data}. Applying Proposition~\ref{Exp-Pos} with $\beta=1/2$, we have
\[u\geqslant\sigma_{0}M\quad\text{a.e.}\quad B_{2\eta\rho}(x_1),\]
provided 
\begin{equation}\label{est6-5}
5^{-\alpha}{\rm Tail}(u_-;x_1,\tfrac{35}8\eta\rho,5\eta\rho)<\sigma_0M.
\end{equation}
Next, we use Proposition~\ref{Exp-Pos} ($n-1$)-times with $\beta=1$ to obtain
\[u\geqslant{\sigma_{0}}{\sigma_1}^{n-1}M\quad\text{a.e.}\quad B_{2^n\eta\rho}(x_1),\]
provided 
\begin{equation}\label{est6-6}
5^{-\alpha}{\rm Tail}(u_-;x_1,\tfrac{35}8\times 2^{n-1}\eta\rho,5\times2^{n-1}\eta\rho)<\sigma_0\sigma_1^{n-1}M.
\end{equation}
Let $n$ be so large that
\[2^{n-1}\eta\leqslant 3\leqslant 2^n\eta.\]
Set $\sigma:=\min\{\sigma_{0},\sigma_1\}$. Since $x_1\in B_\rho,~2^n\eta\leqslant 2\times2^{n-1}\eta\leqslant6$ and $0<\rho<\frac{R}{32}$, we know $B_{2\rho}\subset B_{3\rho}(x_1)\subset B_{2^n\eta\rho}(x_1)\subset B_R$.
On the other hand, we have
\[u\geqslant\sigma^{n}M=2^{n\log_2\sigma}M\geqslant\left(\frac{6}{\eta}\right)^{\log_2\sigma}M
\geqslant(\frac{6}{\widetilde\eta})^{\log_2\sigma}\beta^{-\frac{3}{\alpha}\log_2\sigma}M\quad\text{a.e.}\quad B_{2\rho},\]
where we used (\ref{est44}).
Otherwise, due to $0<\rho<\frac{R}{32}$ and $x_1\in B_\rho$, we have
\begin{equation*}
B_{\frac{35}8\times2^{k-1}\eta\rho}(x_1)\subset B_{\frac78R}\quad\text{and}\quad B_{2^{k-1}\eta\rho}(x_1)\subset B_R,\quad\text{for}~k=1,2,\dots,n.
\end{equation*}
Since $u\geqslant 0$ in $B_R$, we also have
\begin{equation*}
5^{-\alpha}{\rm Tail}(u_-;x_1,\tfrac{35}8\times 2^{k-1}\eta\rho,5\times2^{k-1}\eta\rho)\leqslant 3^\alpha\left(\frac{\rho}{R}\right)^\alpha{\rm Tail}(u_-;\tfrac78R,R),\quad\text{for}~k=1,2,\dots,n.
\end{equation*}
Therefore, if $3^\alpha\xi\beta^q\leqslant\sigma^n$, then (\ref{est6-5}) and (\ref{est6-6}) hold.
Take $q=-\frac{3}{\alpha}\log_2\sigma\gg 1,\xi=\frac{1}{3^\alpha}\big(\frac{6}{\widetilde\eta}\big)^{\log_2\sigma}$ to conclude the desired estimate.
\end{proof}
Now we are ready to prove the weak Harnack inequality, which examines the local integral contribution to the positivity of a local weak super-solution of (\ref{equa}).
\begin{proof}[Proof of Theorem~\ref{weak-harn}]
Let
\[I:=\inf_{B_{2r}(x_0)}u+\left(\frac{r}{R}\right)^\alpha{\rm Tail}(u_-;\tfrac78R,R).\]
Proposition~\ref{Exp-Pos-im} yields that there exist $\xi=\xi({\it data})\in(0,1)$ and $q=q({\it data})>1$, such that either
\[\inf_{B_{2r}(x_0)}u\geqslant{\xi}\bigg(\frac{|\{u\geqslant\lambda\}\cap B_r(x_0)|}{|B_r(x_0)|}\bigg)^q\lambda\]
or
\[\left(\frac{r}{R}\right)^\alpha{\rm Tail}(u_-;\tfrac78R,R)\geqslant{\xi}\bigg(\frac{|\{u\geqslant\lambda\}\cap B_r(x_0)|}{|B_r(x_0)|}\bigg)^q\lambda\]
for any $\lambda>0$. Therefore, we know
\begin{equation}\label{est11}
I\geqslant {\xi}\bigg(\frac{|\{u\geqslant\lambda\}\cap B_r(x_0)|}{|B_r(x_0)|}\bigg)^q\lambda,\quad\forall\lambda>0.
\end{equation}
By (\ref{est11}), we obtain that
\[|\{u\geqslant\lambda\}\cap B_r(x_0)|\leqslant\xi^{-\frac{1}{q}}I^{\frac{1}{q}}\lambda^{-\frac{1}{q}}|B_r(x_0)|.\]
For $\varepsilon\in(0,1)$ to be chosen, we compute using the above estimate:
\begin{align*}
\int_{B_r(x_0)}u^\varepsilon\,dx=&\int_0^\infty |\{u\geqslant\lambda\}\cap B_r(x_0)|\,d\lambda^\varepsilon\\
\leqslant&~I^\varepsilon|B_r|+\int_I^\infty\varepsilon\lambda^{\varepsilon-1}|\{u\geqslant\lambda\}\cap B_r(x_0)|\,d\lambda\\
\leqslant&~I^\varepsilon|B_r|+\xi^{-\frac{1}{q}}I^{\frac{1}{q}}|B_r|\int_{I}^\infty\varepsilon\lambda^{\varepsilon-1-\frac{1}{q}}\,d\lambda.
\end{align*}
The improper integral converges if $\varepsilon<1/q$, and the second term is bounded by
$$\bar{c}I^\varepsilon|B_r|,\quad\text{where}~\bar{c}=\frac{q\varepsilon\xi^{-\frac{1}{q}}}{1-q\varepsilon}.$$
Hence,
\[\bigg(\fint_{B_r(x_0)}u^{\varepsilon}\,dx\bigg)^{\frac{1}{\varepsilon}}\leqslant (\bar{c}+1)^{\frac{1}{\varepsilon}}I.\]
The proof is completed.
\end{proof}

\section{Local boundedness}

\subsection{Local boundedness for local weak sub-solutions}
The main purpose this section is to prove Theorem~\ref{bounded2}. That is a local boundedness estimate for non-negative weak sub-solutions of (\ref{equa}).
In \cite{schu}, in order to draw an analogous conclusion, the author imposed a kind of point-wise upper bound condition on $K(x,y)$, which is also used in \cite{kass2}. In our work, the point-wise upper bound condition has been replaced by (\ref{D}) and (\ref{UJS}) (see Lemma~\ref{point-upper}).
\begin{proof}[Proof of Theorem~\ref{bounded2}]
Upon a translation, we may assume $x_0=0$ and define for $n\in\mathbb{N}_0$,
\begin{align*}
&r_n=\frac{r}{2}+\frac{r}{2^{n+1}},~~\widetilde{r}_n=\frac{r_n+r_{n+1}}{2},~~\widehat{r}_n=\frac{3r_n+r_{n+1}}{4},\\
&B_n=B_{r_n},~~\widetilde{B}_n=B_{\widetilde{r}_n},~~\widehat{B}_n=B_{\widehat{r}_n},\\
&k_n=k-\frac{k}{2^n},~~\widetilde{k}_n=\frac{k_n+k_{n+1}}{2},~~~k\in\mathbb{R}^+,\\
& w_n=(u-k_n)_+,~~\widetilde{ w}_n=(u-\widetilde{k}_n)_+.
\end{align*}
It is helpful to observe that
\[B_{n+1}\subset\widetilde{B}_n\subset\widehat{B}_n\subset B_n.\]
Take $\zeta_n:\mathbb{R}^N\rightarrow\mathbb{R}$ to be a function such that $\text{supp}\,\zeta_n(\cdot)\subset \widetilde{B}_n,~\|\zeta_n\|_\infty\leqslant 1,~\zeta_n(x)=1$ in $B_{n+1}$ and $|\nabla\zeta|\leqslant\gamma(N)2^{n+3}r^{-1}$. According to assumption (\ref{D}) and Lemma~\ref{D->B}, we have
\begin{equation}\label{est6-4}
\sup_{x\in B_n}\int_{B_n}(\zeta_n(y)-\zeta_n(x))^2\,\mu(x,dy)\leqslant\gamma(\mathfrak{C}_2,N,\alpha)\left(\frac{r}{2^{n+3}}\right)^{-\alpha}.
\end{equation}
By using Proposition~\ref{DG-class} we get
\begin{align}\label{est30}
\mathcal{E}_{\mathbb{R}^N}^\mu(\widetilde{w}_n\zeta_n,\widetilde{w}_n\zeta_n)
\leqslant&~2\bigg[\int_{B_n}\int_{B_n}\max\{\widetilde{ w}_n^2(x),\widetilde{ w}_n^2(y)\}(\zeta_n(x)-\zeta_n(y))^2\,\mu(x,dy)dx\nonumber\\
&~+\int_{B_n}\widetilde{ w}_n(x)\zeta_n^2(x)\,dx\bigg(\underset{x\in \text{supp}\,\zeta_n}{\text{sup}}\int_{\mathbb{R}^N\backslash B_n}\widetilde{ w}_n(y)\,\mu(x,dy)\bigg)\bigg]\nonumber\\
&~+2\int_{B_n}\int_{\mathbb{R}^N\backslash B_n}|\widetilde{w}_n(x)\zeta_n(x)|^2\,\mu(x,dy)dx.
\end{align}
By the assumption (\ref{est6-4}), we obtain the following estimate for the first term in the right hand-side of the inequality above,
\begin{align}\label{est31}
\int_{B_n}\int_{B_n}\widetilde{ w}_n^2(x)(\zeta_n(x)-\zeta_n(y))^2\,\mu(x,dy)dx\leqslant&~\gamma(\mathfrak{C}_2,N,\alpha) r^{-\alpha}2^{\alpha(n+3)}\int_{B_n}\widetilde{ w}_n^2(x)\,dx\nonumber\\
\leqslant&~\gamma(\mathfrak{C}_2,N,\alpha)2^{\alpha n}r_n^{N-\alpha}\fint_{B_n}\widetilde{ w}_n^2(x)\,dx,
\end{align}
where we used $r_n\leqslant r$ in the last line. For the second term on the right of (\ref{est30}),
We note that for $x\in\widehat{B}_n,~y\in\mathbb{R}^N\backslash B_n,$
\[\frac{|y|}{|x-y|}\leqslant\frac{|x-y|+|x|}{|x-y|}\leqslant 1+\frac{\hat{r}_n}{r_n-\hat{r}_n}\quad\Longrightarrow \quad \frac{1}{|x-y|^{N+\alpha}}\leqslant\left(\frac{{r}_n}{r_n-\hat{r}_n}\right)^{N+\alpha}\frac{1}{|y|^{N+\alpha}}.\]
By using Lemma~\ref{ujs-control} and Lemma~\ref{point-upper}, we have
\begin{align}\label{est32}
&\int_{B_n}\widetilde{ w}_n(x)\zeta_n^2(x)\,dx\bigg(\underset{x\in \text{supp}\,\zeta_n}{\text{sup}}\int_{\mathbb{R}^N\backslash B_n}\widetilde{ w}_n(y)\,\mu(x,dy)\bigg)\nonumber\\
&~\qquad\leqslant\gamma \int_{B_n}\widetilde{w}_n(x)\zeta_n^2(x)\,dx\left(\frac{r}{2^{n+3}}\right)^{-N}\int_{\widehat{B}_n}\int_{\mathbb{R}^N\backslash B_n} w_0(y)\,\mu(x,dy)dx\nonumber\\
&~\qquad\leqslant\gamma \int_{B_n}\widetilde{w}_n(x)\zeta_n^2(x)\,dx\left(\frac{r}{2^{n+3}}\right)^{-N}\int_{\widehat{B}_n}\int_{\mathbb{R}^N\backslash B_n}\frac{w_0(y)}{|x-y|^{N+\alpha}}\,dydx\nonumber\\
&~\qquad\leqslant\gamma\left(\frac{r}{2^{n+3}}\right)^{-N}\int_{B_n}\widetilde{w}_n(x)\zeta_n^2(x)\,dx\Big(\frac{{r}_n}{r_n-\hat{r}_n}\Big)^{N+\alpha}\hat{r}_n^N\int_{\mathbb{R}^N\backslash B_n}\frac{w_0(y)}{|y|^{N+\alpha}}\,dy\nonumber\\
&~\qquad\leqslant\gamma 2^{nN}r^{-2N-\alpha}r_n^{2N+\alpha}\int_{B_n}\widetilde{w}_n(x)\zeta_n^2(x)\,dx\int_{\mathbb{R}^N\backslash B_n}\frac{(u-k)_+}{|y|^{N+\alpha}}\,dy\nonumber\\
&~\qquad\leqslant\gamma 2^{nN}r^{-2N-\alpha}r_n^{3N+\alpha}\fint_{B_n}\frac{w_n^2(x)}{\widetilde{k}_n-k_n}\,dx\int_{\mathbb{R}^N\backslash B_{r/2}}\frac{u_+(y)}{|y|^{N+\alpha}}\,dy\nonumber\\
&~\qquad\leqslant\gamma 2^{nN}r^\alpha r_n^{N-\alpha} k^{-1}\fint_{B_n}w_n^2(x)\,dx\int_{\mathbb{R}^N\backslash B_{r/2}}\frac{u_+(y)}{|y|^{N+\alpha}}\,dy,
\end{align}
where $\gamma$ depend on $\mathfrak{C}_2,N,\alpha$ and $\hat{c}$, where we have used the definitions of $B_n,$ $\widetilde{B}_n,$ $\widehat{B}_n,$ $k_n,$ $\widetilde{k}_n,$ $w_n,$ $\widetilde w_n$ and the facts $\widetilde{ w}_n\leqslant{ w_n^2}/(\widetilde{k}_n-k_n)$.
For the last term on the right of (\ref{est30}), we have
\begin{align*}
\int_{B_n}\int_{\mathbb{R}^N\backslash B_n}|\widetilde{w}_n(x)\zeta_n(x)|^2\,\mu(x,dy)dx\leqslant&~\underset{x\in\widetilde{B}_n}{\text{ sup}}\int_{\mathbb{R}^N\backslash B_n}\mu(x,dy)\int_{B_n}\widetilde{w}_n^2(x)\,dx\\
\leqslant&~\gamma(\mathfrak{C}_2,\alpha)2^{n\alpha}r^{-\alpha}\int_{B_n}\widetilde{w}_n^2(x)\,dx\\
\leqslant&~\gamma(\mathfrak{C}_2,\alpha)2^{n\alpha}r_n^{N-\alpha}\fint_{B_n}\widetilde{w}_n^2(x)\,dx.
\end{align*}
The left-hand side of (\ref{est30}) can be estimated by using Lemma~\ref{poin->FK} with $v=\widetilde{w}_n\zeta_n$ and $\Omega_0=\widetilde{B}_n\cap\{u\geqslant\widetilde{k}_n\}$ as follow,
\begin{align}\label{est45}
c_0\mathcal{E}_{\mathbb{R}^N}^\mu(\widetilde{w}_n\zeta_n,\widetilde{w}_n\zeta_n)\geqslant&~|\widetilde{B}_n\cap\{u\geqslant\widetilde{k}_n\}|^{-\frac{\alpha}{N}}\int_{\widetilde{B}_n\cap\{u\geqslant\widetilde{k}_n\}}|\widetilde{w}_n\zeta_n|^2\,dx\nonumber\\
\geqslant&~|{B}_n\cap\{u\geqslant\widetilde{k}_n\}|^{-\frac{\alpha}{N}}\int_{B_{n+1}}w_{n+1}^2\,dx,
\end{align}
where we use the definitions of $\widetilde{w}_n,\zeta_n$ and the fact $B_{n+1}\subset\widetilde{B}_n,w_{n+1}\leqslant\widetilde{w}_n$. Since
\begin{align*}
|B_n\cap\{u\geqslant\widetilde{k}_n\}|\left(\frac{k}{2^{n+2}}\right)^2=&\int_{B_n\cap\{u\geqslant\widetilde{k}_n\}}(\widetilde{k}_n-k_n)^2\,dx\nonumber\\
\leqslant&\int_{B_n\cap\{u\geqslant\widetilde{k}_n\}}(u-k_n)^2\,dx\\
\leqslant&\int_{B_n}(u-k_n)^2\,dx,
\end{align*}
we have
\begin{equation}\label{est46}
|B_n\cap\{u\geqslant\widetilde{k}_n\}|^{-\frac{\alpha}{N}}\geqslant\bigg(\frac{2^{n+2}}{k}\bigg)^{-\frac{2\alpha}{N}}\bigg(\int_{B_n}(u-k_n)^2\,dx\bigg)^{-\frac{\alpha}{N}}.
\end{equation}
Inserting (\ref{est46}) into (\ref{est45}) implies
\begin{align}\label{est47}
c_0\mathcal{E}_{\mathbb{R}^N}^\mu(\widetilde{w}_n\zeta_n,\widetilde{w}_n\zeta_n)\geqslant&~\bigg(\frac{2^{n+2}}{k}\bigg)^{-\frac{2\alpha}{N}}\bigg(\int_{B_n}(u-k_n)^2\,dx\bigg)^{-\frac{\alpha}{N}}\int_{B_{n+1}}w_{n+1}^2\,dx\nonumber\\
\geqslant&~\gamma(N,\alpha)k^{\frac{2\alpha}{N}}2^{-\frac{2\alpha}{N}n}r_n^{N-\alpha}\bigg(\fint_{B_n}(u-k_n)^2\,dx\bigg)^{-\frac{\alpha}{N}}\fint_{B_{n+1}}w_{n+1}^2\,dx
\end{align}
where we used
\[\frac{r_{n+1}}{r_n}=\frac{2^{n+1}+1}{2^{n+1}+2}>\frac{1}{2}.\]

Setting $$Y_n:=\left(\fint_{B_n} w_n^2(x)\,dx\right)^{\frac{1}{2}}$$ and combining (\ref{est47}) (\ref{est30}), (\ref{est31}), (\ref{est32}), we obtain
\[Y_{n+1}^2\left(\frac{k}{2^n}\right)^{\frac{2\alpha}{N}}\leqslant\gamma Y_n^{2(1+\frac{\alpha}{N})}\bigg[2^{nN} k^{-1}\int_{\mathbb{R}^N\backslash B_{r/2}}\frac{u_+(y)}{|y|^{N+\alpha}}\,dy+2^{\alpha n}+1\bigg]\]
for some $\gamma=\gamma({\it data},\hat{c})$. Now, by taking
\begin{equation}\label{kk}
k\geqslant\delta\int_{\mathbb{R}^N\backslash B_{r/2}}\frac{u_+(y)}{|y|^{N+\alpha}}\,dy,~~\delta\in(0,1],
\end{equation}
we get
\[\frac{Y_{n+1}}{k}\leqslant\delta^{-\frac{1}{2}}\gamma^{\frac{1}{2}}2^{(\frac{\alpha+N}{2}+\frac{\alpha}{N})n}\left(\frac{Y_n}{k}\right)^{1+\frac{\alpha}{N}}.\]
By the fast geometric convergence, cf.~\cite{dibe-gia-ves}, if we require that
\begin{equation*}
\frac{Y_0}{k}\leqslant\delta^{\frac{N}{2\alpha}}H^{-1},
\end{equation*}
with
\[H:=\gamma^{\frac{N}{2\alpha}}2^{\frac{N^2(\alpha+N)}{2\alpha^2}+\frac{N}{\alpha}},\]
then $\frac{Y_n}{k}\rightarrow 0$ as $n\to\infty$.
To satisfy the above condition, we choose
\[k=\delta\int_{\mathbb{R}^N\backslash B_{r/2}}\frac{u_+(y)}{|y|^{N+\alpha}}\,dy+\delta^{-\frac{N}{2\alpha}}HY_0,\]
which also satisfies (\ref{kk}).
\par
Consequently, we deduce
\begin{equation*}
\sup_{B_{\frac{r}{2}}}u\leqslant k\leqslant\delta\int_{\mathbb{R}^N\backslash B_{r/2}}\frac{u_+(y)}{|y|^{N+\alpha}}\,dy+\delta^{-\frac{N}{2\alpha}}H\left(\fint_{B_r}u_+^2\,dx\right)^{\frac{1}{2}}.
\end{equation*}
This completes the proof.
\end{proof}
\subsection{Local boundedness for local weak solutions}\label{loc-boun-super}
In this subsection, we aim to introduce another version of the locally bounded theorem for local weak solutions, see Theorem~\ref{bounded}, which is an important prior estimate for the establishment of Harnack inequality. To prove that, we first introduce a useful lemma. 
\begin{lemma}\label{pre-control}
Let $\mu$ satisfy (\ref{D}). Let $B_R(x_0)\subset E$. Let $u$ be a local weak super-solution of (\ref{equa}) in $E$, such that $u\geqslant 0$ in $B_R(x_0)$. Then, for $0<\rho<\frac{7}{8}R,~\ell\in(\frac{1}{4},1)$, we have
\begin{equation*}
\int_{B_{l\rho}(x_0)}\int_{\mathbb{R}^N\backslash B_\rho(x_0)}u_+(y)\,\mu(x,dy)dx\leqslant \gamma \rho^{N-\alpha}\bigg(\sup_{B_\rho(x_0)}u+\bigg(\frac{\rho}{R}\bigg)^\alpha{\rm Tail}(u_-;x_0,\tfrac78R,R)\bigg),
\end{equation*}
where the constant $\gamma$ depends only on $\mathfrak{C}_2,\alpha$ and $N$.
\end{lemma}
\begin{proof}
For simplicity, we assume $x_0=0.$ Set $k=\underset{B_\rho}{\sup}~u$ and $w=u-2k$. Take $\zeta:\mathbb{R}^N\rightarrow\mathbb{R}$ to be a function such that $\text{supp}\,\zeta(\cdot)\subset B_{\frac{1+\ell}{2}\rho},~\|\zeta\|_\infty\leqslant 1,~\zeta(x)=1~\text{in}~B_{\ell\rho}$ and $|\nabla\zeta|\leqslant\frac{\gamma(N)}{(1-\ell)\rho}$. According to assumption (\ref{D}) and Lemma~\ref{D->B}, we have
\begin{align}\label{est6-3}
&\sup_{x\in B_R}\int_{B_R}(\zeta(y)-\zeta(x))^2\,\mu(x,dy)\leqslant 2^\alpha \mathfrak{C}_2\left[(1-\ell)\rho\right]^{-\alpha}.
\end{align}
By Proposition~\ref{DG-class}, we have
\begin{align}\label{est38}
&\int_{B_R}w_-(x)\zeta^2(x)\,dx\bigg(\int_{\mathbb{R}^N}w_+(y)\,\mu(x,dy)\bigg)\nonumber\\
&~\qquad\leqslant\gamma\int_{B_R}\int_{B_R}\max\{w_-^2(x),w_-^2(y)\}(\zeta(x)-\zeta(y))^2\,\mu(x,dy)dx\nonumber\\
&~\quad\qquad+\gamma\int_{B_R}w_-(x)\zeta^2(x)\,dx\bigg(\underset{x\in\text{supp}\,\zeta}{\text{sup}}\int_{\mathbb{R}^N\backslash B_R}u_-(y)\,\mu(x,dy)\bigg)\nonumber\\
&~\qquad=:I_1+I_2.
\end{align}
Noticing $w_-\leqslant 2k$, together with (\ref{est6-3}), we can deduce
\begin{equation}\label{est39}
I_1\leqslant 2^{\alpha+2}\gamma \mathfrak{C}_2k^2\rho^N[(1-\ell)\rho]^{-\alpha},
\end{equation}
and
\begin{equation}\label{est40}
I_2\leqslant 2\gamma k\rho^N R^{-\alpha}{\rm Tail}(u_-;\tfrac78R,R).
\end{equation}
\par
Next, we treat the left-hand side of (\ref{est38}). We know that $w_+=0$ in $B_\rho$ from the definition of $k$, thus,
\begin{equation*}
\int_{\mathbb{R}^N}w_+(y)\,\mu(x,dy)=\int_{\mathbb{R}^N\backslash B_\rho}w_+(y)\,\mu(x,dy).
\end{equation*}
Recalling the definition of $w$, we have
\begin{align}\label{est41}
\int_{B_R}w_-(x)\zeta^2(x)\,dx\bigg(\int_{\mathbb{R}^N}w_+(y)\,\mu(x,dy)\bigg)\geqslant&\int_{B_{\ell\rho}}w_-(x)\zeta^2(x)\,dx\bigg(\int_{\mathbb{R}^N\backslash B_\rho}w_+(y)\,\mu(x,dy)\bigg)\nonumber\\
\geqslant&~k\int_{B_{\ell\rho}}\int_{\mathbb{R}^N\backslash B_\rho}(u_+(y)-2k)\,\mu(x,dy)dx.
\end{align}
Combining (\ref{est38}), (\ref{est39}), (\ref{est40}) and (\ref{est41}), we get
\begin{align*}
&\int_{B_{\ell\rho}}\int_{\mathbb{R}^N\backslash B_\rho}u_+(y)\,\mu(x,dy)dx\\
&~\quad\leqslant2k\int_{B_{\ell\rho}}\int_{\mathbb{R}^N\backslash B_\rho}\mu(x,dy)dx+\gamma\Big[k\rho^N[(1-\ell)\rho]^{-\alpha}+\rho^N R^{-\alpha}{\rm Tail}(u_-;7R/8,R)\Big]\\
&~\quad\leqslant\gamma k [(1-\ell)\rho]^{-\alpha}|B_{\ell\rho}|+\gamma\Big[k\rho^N[(1-\ell)\rho]^{-\alpha}+\rho^N R^{-\alpha}{\rm Tail}(u_-;\tfrac78R,R)\Big]\\
&~\quad\leqslant\gamma\rho^{N-\alpha}\left[k+\left(\frac{\rho}{R}\right)^{\alpha}{\rm Tail}(u_-;\tfrac78R,R)\right]
\end{align*}
for some $\gamma=\gamma(\mathfrak{C}_2,N,\alpha)$, where we used the assumption (\ref{D}), Lemma~\ref{D2} and $\ell\in(\frac{1}{2},1)$. This is the desired result, recalling the definition of $k$.
\end{proof}

\begin{theorem}\label{bounded}
Let $\mu$ be represented by (\ref{mu}) and satisfy (\ref{S}), (\ref{A-poin}), (\ref{D}) and (\ref{UJS}). Let $u$ be a local weak solution of (\ref{equa}) in $E$, such that $u\geqslant 0$ in $B_R(x_0)$. Then, for any $0<r<R$, there holds
\begin{equation*}\label{bounded3}
\sup_{B_{r/2}(x_0)}u\leqslant\delta\bigg[\sup_{B_r(x_0)}u+\left(\frac{r}{R}\right)^\alpha{\rm Tail}(u_-;x_0,\tfrac78R,R)\bigg]+\gamma\delta^{-\frac{N}{2\alpha}}\bigg(\fint_{B_r(x_0)}u_+^2\,dx\bigg)^{\frac{1}{2}},
\end{equation*}
for any $\delta\in(0,1]$, provided the ball $B_R(x_0)$ is included in $E$, where the constant $\gamma>1$ depends only on {\it data} and $\hat{c}$.
\end{theorem}
\begin{proof}
We closely follow the arguments in the proof of Theorem~\ref{bounded2}. By using Lemma~\ref{ujs-control} and Lemma~\ref{pre-control}, we revise the estimation (\ref{est32}) as follows
\begin{align*}
&\int_{B_n}\widetilde{ w}_n(x)\zeta_n^2(x)\,dx\bigg(\underset{x\in \text{supp}\,\zeta_n}{\text{sup}}\int_{\mathbb{R}^N\backslash B_n}\widetilde{ w}_n(y)\,\mu(x,dy)\bigg)\nonumber\\
&~\qquad\leqslant\int_{B_n}\widetilde{ w}_n(x)\zeta_n^2(x)\,dx\bigg(\underset{x\in \widetilde{B}_n}{\text{sup}}\int_{\mathbb{R}^N\backslash B_n} w_0(y)\,\mu(x,dy)\bigg)\nonumber\\
&~\qquad\leqslant\gamma \int_{B_n}\widetilde{w}_n(x)\zeta_n^2(x)\,dx\left(\frac{r}{2^{n+3}}\right)^{-N}\int_{\widehat{B}_n}\int_{\mathbb{R}^N\backslash B_n} w_0(y)\,\mu(x,dy)dx\nonumber\\
&~\qquad\leqslant\gamma r_n^{N-\alpha}\int_{B_n}\widetilde{ w}_n(x)\zeta_n^2(x)\,dx\left(\frac{r}{2^{n+3}}\right)^{-N}\bigg[\sup_{B_n}u+\left(\frac{r_n}{R}\right)^\alpha{\rm Tail}(u_-;\tfrac78R,R)\bigg]\nonumber\\
&~\qquad\leqslant\gamma r_n^{2N-\alpha}\left(\frac{r}{2^{n+3}}\right)^{-N}\fint_{B_n}\frac{ w_n^2(x)}{\widetilde{k}_n-k_n}\,dx\bigg[\sup_{B_r}u+\left(\frac{r}{R}\right)^\alpha{\rm Tail}(u_-;\tfrac78R,R)\bigg]\nonumber\\
&~\qquad\leqslant\gamma 2^{(n+2)N} r_n^{N-\alpha}k^{-1}\fint_{B_n} w_n^2(x)\,dx\bigg[\sup_{B_r}u+\left(\frac{r}{R}\right)^\alpha{\rm Tail}(u_-;\tfrac78R,R)\bigg],
\end{align*}
where $\gamma$ depend on $\mathfrak{C}_2,N,\alpha$ and $\hat{c}$. The difference from (\ref{est32}) is that
$$\int_{\mathbb{R}^N\backslash B_{r/2}}\frac{u_+(y)}{|y|^{N+\alpha}}\,dy$$ is replaced by $$\sup_{B_r}u+\left(\frac{r}{R}\right)^\alpha{\rm Tail}(u_-;\tfrac78R,R).$$
The rest of the proof is the same as the proof of Theorem~\ref{bounded2}.
\end{proof}

\section{Full Harnack inequality}
Based on the locally boundedness estimates in Section~6 and the weak Harnack inequality from Section~\ref{weak-Harn}, a full Harnack inequality can be established through a standard iterative process.

\begin{proof}[Proof of Theorem~\ref{full-harn}]
    The estimate in Theorem~\ref{bounded} yields
    \begin{equation}\label{est36}
        \sup_{B_{\frac{\rho}{2}}(x_0)}u\leqslant\gamma\frac{1}{\delta^{\frac{N}{2\alpha}}}\bigg(\fint_{B_\rho(x_0)}u_+^2\,dx\bigg)^{\frac{1}{2}}+\delta\sup_{B_\rho(x_0)}u+\delta2^{-\alpha} {\rm Tail}(u_-;x_0,\tfrac74\rho,2\rho),
    \end{equation}
    for any $B_R(x_0)\subset E$ and $0<\rho<\frac12R$. To proceed, we first need to perform a covering argument. Let $\frac{1}{2}\leqslant\sigma^\prime<\sigma\leqslant 1$. Note that
    \begin{equation*}
    B_{\sigma^\prime r}\subseteq\underset{z\in B_{\sigma^\prime r}}{\bigcup}B_{\frac{\sigma-\sigma^\prime}{2}r}(z)~~\text{and}~~B_{(\sigma-\sigma^\prime)r}(z)\subset B_{\sigma r}~~\text{for any}~~z\in B_{\sigma^\prime r}.
    \end{equation*}
    Otherwise, since $r<\frac{R}{32}$, we have
    \[B_{\frac74r(\sigma-\sigma^\prime)}(z)\subset B_{\frac78R}\quad\text{and}\quad B_{2r(\sigma-\sigma^\prime)}(z)\subset B_R,~~\text{for any}~~z\in B_{\sigma^\prime r}.\]
    Due to $u\geqslant 0$ in $B_R(x_0)$, we also have
    \[2^{-\alpha}{\rm Tail}(u_-;z,\tfrac74(\sigma-\sigma^\prime)r,2(\sigma-\sigma^\prime)r)\leqslant\left(\frac{r}{R}\right)^\alpha{\rm Tail}(u_-;\tfrac78R,R).\]
    Therefore, by using (\ref{est36}) with $\rho=(\sigma-\sigma^\prime)r$ and $x_0=z$, we get
    \begin{align*}
        \sup_{B_{\sigma^\prime r}}u\leqslant&\sup_{z\in B_{\sigma^\prime r}}\Bigg(\sup_{B_{\frac{\sigma-\sigma^\prime}{2}r}(z)}u\Bigg)\nonumber\\
        \leqslant&~\gamma\delta^{-\frac{N}{2\alpha}}\bigg(\fint_{B_{(\sigma-\sigma^\prime)r}(z)}u_+^2\,dx\bigg)^{\frac{1}{2}}
        +\delta\sup_{B_{(\sigma-\sigma^\prime)r}(z)}u+\delta2^{-\alpha}{\rm Tail}(u_-;z,\tfrac74(\sigma-\sigma^\prime)r,2(\sigma-\sigma^\prime)r)\nonumber\\
        \leqslant&~\gamma\delta^{-\frac{N}{2\alpha}}\Big(\frac{\sigma}{\sigma-\sigma^\prime}\Big)^{\frac{N}{2}}\left(\fint_{B_{\sigma r}}u_+^2\,dx\right)^{\frac{1}{2}}+\delta\sup_{B_{\sigma r}}u+\delta2^{-\alpha}{\rm Tail}(u_-;z,\tfrac74(\sigma-\sigma^\prime)r,2(\sigma-\sigma^\prime)r)\nonumber\\
        \leqslant&~\gamma\frac{\delta^{-\frac{N}{2\alpha}}}{(\sigma-\sigma^\prime)^{\frac{N}{2}}}(\sup_{B_{\sigma r}}u)^{\frac{2-\varepsilon}{2}}\left(\fint_{B_{\sigma r}}u_+^\varepsilon\,dx\right)^{\frac{1}{2}}+\delta\sup_{B_{\sigma r}}u+\delta\left(\frac{r}{R}\right)^\alpha{\rm Tail}(u_-;\tfrac78R,R),
    \end{align*}
    where $\varepsilon=\varepsilon({\it data})\in(0,1)$ is the constant from Theorem~\ref{weak-harn}.
    By choosing $\delta=\frac{1}{4}$, a standard application of Young's inequality yields
    \begin{equation}\label{est37}
        \sup_{B_{\sigma^\prime r}}u\leqslant\frac{1}{2}\sup_{B_{\sigma r}}u+\frac{\gamma}{(\sigma-\sigma^\prime)^{\frac{N}{\varepsilon}}}\left(\fint_{B_{r}}u_+^\varepsilon \,dx\right)^{\varepsilon}+\gamma\left(\frac{r}{R}\right)^\alpha{\rm Tail}(u_-;\tfrac78R,R).
    \end{equation}
    Next, consider the sequence $\{t_i\}$ defined by
    \[t_0=\frac{1}{2},\quad t_{i+1}-t_{i}=\frac{1}{2}(1-\tau)\tau^i\]
    with $0<\tau<1$ to be fixed. By (\ref{est37}) we get
    \begin{equation*}
    \sup_{B_{t_i r}}u\leqslant\frac{1}{2}\sup_{B_{t_{i+1} r}}u+\frac{\gamma}{[(1-\tau)\tau^i]^{\frac{N}{\varepsilon}}}\left(\fint_{B_{r}}u_+^\varepsilon \,dx\right)^{\varepsilon}+\gamma\left(\frac{r}{R}\right)^\alpha{\rm Tail}(u_-;\tfrac78R,R),
    \end{equation*}
    $i=0,1,\cdots$ By iteration, we obtain
    \begin{equation*}
    \sup_{B_{t_0 r}}u\leqslant\left(\frac{1}{2}\right)^k\sup_{B_{t_{k} r}}u+\bigg[\frac{2^{\frac{N}{\varepsilon}}\gamma}{(1-\tau)^{\frac{N}{\varepsilon}}}\left(\fint_{B_{r}}u_+^\varepsilon \,dx\right)^{\varepsilon}+\gamma\left(\frac{r}{R}\right)^\alpha{\rm Tail}(u_-;\tfrac78R,R)\bigg]\sum_{i=0}^{k-1}2^{-i}\tau^{-\frac{N}{\varepsilon}i}.
    \end{equation*}
    We choose now $\tau$ such that $\frac{1}{2}\tau^{-\frac{N}{\varepsilon}}<1$ and let $k\to\infty$, to get
    \[\sup_{B_{r/2}}u\leqslant\gamma\left(\fint_{B_r}u_+^\varepsilon\,dx\right)^\varepsilon+\gamma\left(\frac{r}{R}\right)^\alpha{\rm Tail}(u_-;\tfrac78R,R),\]
    where the constant $\gamma$ depends also on $\varepsilon$. Combining the above estimate with the weak Harnack inequality established in Theorem~\ref{weak-harn}, we get
    \begin{equation*}
        \sup_{B_{r/2}}u\leqslant\gamma\inf_{B_{2r}}u+\gamma \left(\frac{r}{R}\right)^\alpha{\rm Tail}(u_-;\tfrac78R,R)\leqslant\gamma\inf_{B_{r/2}}u+\gamma \left(\frac{r}{R}\right)^\alpha{\rm Tail}(u_-;\tfrac78R,R).
    \end{equation*}
    Redefine $r/2$ as $r$ to conclude.
\end{proof}


\begin{thebibliography}{99}
\bibitem{barlow}
M.~T.~Barlow, R.~F.~Bass, T.~Kumagai, {\it Parabolic Harnack inequality and heat kernel estimates for random walks with long range jumps}, Math. Z. {\bf 261}(2009), no. 2, 297-320.
\bibitem{bogdan}
K.~Bogdan, P.~Sztonyk, {\it Harnack's inequality for stable L\'evy processes}, Potential Anal. {\bf 22}(2005), no. 2, 133-150.
\bibitem{cassa}
F.~M.~Cassanello, F.~G.~D\"uzg\"un, A.~Iannizzotto, {\it H\"older regularity for the fractional p-Laplacian, revisited}, preprint, (2024).
\bibitem{kass6}
J.~Chaker, M.~Kassmann, {\it Nonlocal operators with singular anisotropic kernels}, Comm. Partial Differential Equations {\bf 45}(2020), no. 1, 1-31.
\bibitem{chen-kim}
Z.-Q.~Chen, P.~Kim, T.~Kumagai, {\it On heat kernel estimates and parabolic Harnack inequality for jump processes on metric measure spaces}, Acta Math. Sinica, Engl. Ser. {\bf 25}(2009), no. 7, 1067-1086.
\bibitem{chen-kuma1}
Z.-Q.~Chen, T.~Kumagai, J.~Wang, {\it Elliptic Harnack inequalities for symmetric non-local Dirichlet forms}, J. Math. Pures Appl. {\bf 125}(2019), 1-42.
\bibitem{chen-kuma2}
Z.-Q.~Chen, T.~Kumagai, J.~Wang, {\it Stability of parabolic Harnack inequalities for symmetric non-local Dirichlet forms}, J. Eur. Math. Soc. {\bf 22}(2020), no. 11, 3747-3803.
\bibitem{cozzi}
M.~Cozzi, {\it Regularity results and Harnack inequalities for minimizers and solutions of nonlocal problems: a unified approach via fractional De Giorgi classes}, J. Funct. Anal. {\bf 272}(2017), no. 11, 4762-4837.
\bibitem{dibe-trud}
E.~DiBenedetto, N.~Trudinger, {\it Harnack inequalities for quasiminima of variational integrals}, Ann. Inst. H. Poincar\'e Anal. Non Lin\'eaire {\bf 1}(1984), no. 4, 295-308.
\bibitem{dibe-gia-ves}
E.~DiBenedetto, U.~Gianazza, V.~Vespri, {\it``Harnack's Inequality for Degenerate and Singular Parabolic Equations"}, Springer Monographs in Mathematics, Springer-Verlag, New York. (2012)
\bibitem{casrto1}
A.~Di Castro, T.~Kuusi, G.~Palatucci, {\it Nonlocal Harnack inequalities}, J. Funct. Anal. {\bf 267}(2014), no. 6, 1807-1836.
\bibitem{casrto}
A.~Di Castro, T.~Kuusi, G.~Palatucci, {\it Local behavior of fractional p-minimizers}, Ann. Inst. H. Poincar\'e C, Anal. Non Lin\'eaire. {\bf 33}(2016), no. 5, 1279-1299.
\bibitem{duzgun}
F.~G.~D\"uzg\"un, A.~Iannizzotto, V.~Vespri, {\it A clustering theorem in fractional sobolev spaces}, \textit{arXiv:2305.19965} (2023).
\bibitem{kass1}
B.~Dyda, M.~Kassmann, {\it Regularity estimates for elliptic nonlocal operators}, Anal. PDE. {\bf 13}(2020), no. 2, 317-370.
\bibitem{hu-yu}
J.~Hu, Z.~Yu, {\it The weak elliptic Harnack inequality revisited}, Asian J. Math. {\bf 27}(2023), no. 5, 771-828.
\bibitem{kass3}
M.~Kassmann, {\it A priori estimates for integro-differential operators with measurable kernels}, Calc. Var. Partial Differential Equations. {\bf 34}(2009), no. 1, 1-21.
\bibitem{kass5}
M.~Kassmann, M.~Weidner, {\it Upper heat kernel estimates for nonlocal operators via Aronson's method}, Calc. Var. Partial Differential Equations. {\bf 62}(2023), no. 2, Paper No. 68, 27 pp.
\bibitem{kass2}
M.~Kassmann, M.~Weidner, {\it The parabolic Harnack inequality for nonlocal equations}, \textit{arXiv:2303.05975} (2023).

\bibitem{liao2}
N.~Liao, {\it H\"older regularity for parabolic fractional p-Laplacian}, Calc. Var. Partial Differential Equations. {\bf 63}(2024), no. 1, Paper No. 22, 34 pp.
\bibitem{sal}
L. Saloff-Coste, {\it A note on Poincar\'e, Sobolev, and Harnack inequalities}, Inter. Math. Res. Notices. (1992), no. 2, 27-38.
\bibitem{schu}
T.~Schulze, {\it Nonlocal operators with symmetric kernels}, PhD Thesis (Universit\"at Bielefeld), (2020).
\end{thebibliography}
\end{document}